\documentclass[reqno]{amsart}
\usepackage[top=2cm, bottom=2cm, outer=3cm, inner=3cm]{geometry}
\usepackage{amssymb,color}
\usepackage{amsmath}
\usepackage{amsfonts}
\usepackage{fouridx}
\usepackage[normalem]{ulem}
\usepackage{soul}
\usepackage[pdftex]{hyperref}
\usepackage{relsize}
\usepackage[utf8]{inputenc}
\usepackage{mathrsfs}
\usepackage[english]{babel}
\usepackage[totoc]{idxlayout}
\usepackage{fancyhdr}
\usepackage{paralist}
\usepackage{xcolor}
\usepackage[mathscr]{euscript}
\usepackage{amsmath, amssymb, amsthm}
\usepackage{mathrsfs}
\usepackage{mathtools}
\usepackage{mathtools}

\usepackage{todonotes}

\newcommand{\A}{\mathcal{A}}
\newcommand{\V}{\mathcal{V}}
\newcommand{\Phat}{\mathbb{P}}
\newcommand{\Z}{\mathcal{Z}}
\newcommand{\EE}{\mathbb{E}}
\newcommand{\PP}{\mathbb{P}}
\newcommand{\FF}{\mathcal{F}}
\newcommand{\NN}{\mathbb{N}}
\newcommand{\VV}{\mathcal{V}}
\newcommand{\LL}{\mathcal{L}}
\newcommand{\DD}{\mathcal{D}}
\newcommand{\CC}{\mathcal{C}}
\newcommand{\MM}{\mathcal{M}}

\newcommand{\inner}[2]{\langle #1, #2 \rangle}

\allowdisplaybreaks[3]
\newcommand\delc[1]{}

\newcommand\comcd[1]{}

\newcommand\del[1]{}
\newcommand\deln[1]{}
\newcommand\delr[1]{}

\newcommand\comad[1]{}

\newcommand\Greendel[1]{}
\newcommand\old[1]{}

\numberwithin{equation}{section}

\def\R{{\mathbb R}\,}
\def\N{{\mathbb N}\,}

\def\E{{\mathbb E}\,}

\newcommand{\D}{{\mathscr D}}
\renewcommand{\H}{{\mathcal{H}}}

\def\old#1{}

\def\text#1{{\rm #1}}
\def\newold#1{}

\theoremstyle{plain}
\newtheorem{theorem}{Theorem}[section]
\theoremstyle{remark}

\theoremstyle{plain}
\newtheorem{corollary}[theorem]{Corollary}
\newtheorem{lemma}[theorem]{Lemma}
\newtheorem{proposition}[theorem]{Proposition}
\newtheorem{definition}[theorem]{Definition}

\numberwithin{equation}{section}

\DeclarePairedDelimiter{\ang}{\langle}{\rangle}
\DeclarePairedDelimiter{\norm}{\|}{\|}
\DeclarePairedDelimiter{\abs}{|}{|}

\begin{document}
\title[Martingale Solutions of Stochastic Constrained Modified Swift-Hohenberg Equation]{Martingale Solutions of Stochastic Constrained Modified Swift-Hohenberg Equation}

\author{Saeed Ahmed}
\address{Department of Mathematics\\
Sukkur IBA university\\
Sindh Pakistan}
\email{saeed.msmaths21@iba-suk.edu.pk}

\author{Javed  Hussain}
\address{Department of Mathematics\\
Sukkur IBA university\\
Sindh Pakistan}
\email{javed.brohi@iba-suk.edu.pk}

\date{\today}
\begin{abstract}
In this paper, we aim to prove the existence of global Martingale solution to Stochastic Constrained Modified Swift-Hohenberg Equation driven by stratonovich multiplicative noise. This equation belongs to class of amplitude equations which describe the appearance of pattern formation in nature. This structure allows us to work in a Hilbert space framework and to apply a stochastic Galerkin method. The existence proof is based on energy-type estimates, the tightness criterion of Brzeźniak and collaborators, and Jakubowski's generalization of the Skorokhod theorem..
\end{abstract}

\maketitle

\baselineskip 12pt

\section{\textbf{\Large  Introduction}}

In this paper, we consider the following stochastic-constrained Modified Swift-Hohenberg evolution equation with  the driven by multiplicative Stratonovich noise.
 \begin{align}{\label{main_Prb_Intr}}
        du &=\pi _{u}(-\Delta^{2}u+2\Delta u -au - u^{2n-1}) ~dt +  \sum_{k=1}^{N} B_{k}(u) \circ dW_{k}\\ \notag
u(0,x) &= u_{0}.  \notag \\
u(t,x) &= 0, ~~~~~~~~ on ~~~ x \in \partial \mathcal{O} \notag
    \end{align}

Where $\mathcal{O} \subset \mathbb{R}^{2}$ is smooth, continuous and bounded domain, $\left(W_{k}\right)_{k=1}^{N}$  is the  $N$ dimensional, $\mathbb{R}^{N}$-valued {Brownian motion} on the $\left( \Omega, \mathbb{F}, \left(\mathcal{F}_{t}\right)_{t\geq 0}, \mathbb{P}\right) $  filtered probability space. The noise in  Stratonovich form $\sum_{k=1}^{N} B_{k}(u) \circ dW_{k}$ (see \cite{Hussain_2005}) is gradient type on the vector fields $B_{k}(u)$ that are tangents to the manifold ${M} =  \{~ u \in {H}, ~|u|_{\mathcal{H}}^{2}=1~\}. $
In study of pattern formation, modified Swift-Hohenberg equation plays an important role \cite{Maria.B.kania, Qu,Dirk}. Connected with Rayleigh–Bénard convection, it has been employed to address a variety of problems, such as Taylor–Couette flow \cite {PC,Pomeau} and in the study of lasers \cite{Lega}. In addition to other areas of scieite research, it is a valuable tool in material science. This elucidates the surface morphologies during crystal growth \cite{5}, self-assembly processes \cite{6}, and phase transitions \cite{7}. Deterministic forms model regular patterns, whereas stochastic versions incorporate randomness, capturing thermal fluctuations \cite{9} and growth uncertainties \cite{8}. This equation facilitates a comprehensive understanding of intricate material behaviors, enabling advancements in the field of thin-film film deposition \cite{10} and photonic materials   \cite{11}. At that time, the focus was on the global attractor, stability of stationary solutions, and pattern selections of solutions of the deterministic Modified Swift-Hohenberg equation \cite{Peletier1, Peletier2,Peletier3}. However, in recent years, there has been a growing interest in Stochastic Swift-Hohenberg.  Stochastic models are more realistic as noise models the small irregular fluctuations produced by the microscopic effects. The approximation representation of parameterizing manifold and non-Markovian reduced systems for a stochastic Swift–Hohenberg equation with additives was analyzed in \cite{Guo2}. The
results for approximation of manifolds for stochastic Swift–Hohenberg equation with multiplicative noise in Stratonovich sense can been seen in \cite{Dirk, Dirk1, Lin, Swift}. A rigorous error estimation verification of the existence of an amplitude equation for the stochastic Swift–Hohenberg equation was provided by Klepeal et al. \cite{Klepel}. The dynamics and invariant manifolds for a nonlocal stochastic Swift–Hohenberg equation with multiplicative noise were presented in \cite{Guo}. Especially, local and global martingale solution of $2D$ stochastic modified Swift-Hohenberg equations with multiplicative noise and periodic boundry has established in \cite{chen}. Although we have prove the existance and uniqueness of the global solution for stochastic-constrained Modified Swift-Hohenberg evolution equation with  the Stratonovich noise on a Hilbert manifold in \cite{SCMSHE}, but, to our best of knowledge, no prior work has been done on the existanece of Martingale solutions to the same equation. We want to fill this gap.\\

The simple version of above equation (\ref{main_Prb_Intr}) using (\ref{prj}) is

\begin{align}{\label{main_eq_st_Ito-2}}
      du &= \left[-\Delta^{2}u+2 \Delta u+F(u)\right] ~dt +  \sum_{k=1}^{N} B_{k}(u)\circ  dW_{k} \\ \notag 
      u& = u_{0} \notag 
 \end{align}

Where  the function  $ F :  \mathcal{V}\longrightarrow \mathcal{H}$ is a map defined as $F(u)=\|  u\|^{2}_{{H}^{2}_{0}} ~u + 2\|    u\|^{2}_{{H}^{1}_{0}} ~u  +\| u\|^{2n}_{{\mathcal{L}}^{2n}} u- u^{2n-1} $ and  $n \in \mathbb{N}$  (or, in a general sense, a real number such that $n>\frac{1}{2}$) and $u_{0} \in \mathcal{V} \cap {M}$.
 
 We can write the problem (\ref{main_eq_st_Ito-2}) in Ito's form,

\begin{align}{\label{main_eq_st_Ito}}
      du &= \left[-\Delta^{2}u+2 \Delta u+F(u) + \frac{1}{2}\sum_{k=1}^{N} m_{k}(u) \right] ~dt +  \sum_{k=1}^{N} B_{k}(u) dW_{k} \\ \notag 
      u& = u_{0} \notag 
 \end{align}

 Where
\begin{align}{\label{m_K_st}}
        m_{k}(u)=d_{u}B_{k}(B_{k}(u)), ~~~ \forall u \in \mathcal{H},~~~~ \text{and}~~~~k=1,2,3,...N
    \end{align}

Also, for the fixed elements $f_{1}, f_{2}, f_{3} , ... f_{N}$ in $\mathcal{V}$, the map $B_{k}: \mathcal{V}\rightarrow \mathcal{V}$ is defined as:
\begin{align}{\label{B_K_st}}
    B_{k}(u) &= \pi_{u}(f_{k})=f_{k}- \langle f_{k},u\rangle u,~~~~~~k= 1,2,3,...N
\end{align}

Where $\mathcal{V}$ , $\mathcal{E}$, $\mathcal{H}$ are the function spaces defined in (\ref{spaces}).

Our aim to prove the existence of Martingale solutions to the  equation \eqref{main_eq_st_Ito-2}. We will follows the same procedure as used by Da Prato and Zabczyk in \cite{Da Prato} using compactness, tightness of measure, quadratic variations, and Martingale representation theorem.\\

Our paper is organized as follows: Section 1 presents the introduction of the paper. Section 2 is about the Stochastic framework and main results. Section 3 states Faedo Galerkin Approximation of the problem  \eqref{main_eq_st_Ito-2}. Section 4 prove the tightness of measure along with some estimates controlled by expectations. Section 5 shows the convergence of Quadratic Variations. Finally, Section 6 reflects the existance of martingale solution to the main problem \eqref{main_eq_st_Ito-2}.

\section{\textbf{\Large Stochastic framework and main results}}

Assume that  spaces $(\mathcal{E}, \|.\|_\mathcal{E}) $, $(\mathcal{V},\|.\|_{\mathcal{V}}) $  and $(\mathcal{H}, |.|_\mathcal{H}) $ are denoted as
    \begin{eqnarray}  {\label{spaces}}
\mathcal{H}:= \mathcal{\mathcal{L}}^2(\mathcal{O}), ~~
\mathcal{V} := {H}_{0}^{1}(\mathcal{O}) \cap {H}^{2}(\mathcal{O}),~~~\text{and} ~~~~~
\mathcal{E} := \mathcal{D}(\mathcal{A}) =  {H}_{0}^{2}(\mathcal{O}) \cap {H}^4(\mathcal{O}).
\end{eqnarray} 
and we suppose the standard choice for $D(A)$ (with clamped boundary conditions \cite{Evans}), That is:
\begin{eqnarray}
D(A) = \bigl\{ u \in H^4(\mathcal{O}) \mid  u= \partial_n u = 0 \text{ ~on~~ } \partial\mathcal{O} \bigr\}
\end{eqnarray}
where $\partial_n u$ denotes the normal derivative on the boundary.\\
Also, these spaces are dense and continuous, that is, 
\begin{align}{\label{Ass_2.2.2_St}}
    \mathcal{E} \hookrightarrow \mathcal{V} \hookrightarrow  \mathcal{H}
\end{align}
And  we denote the operator 
 $\mathcal{A} :  D(\mathcal{A}) \rightarrow {\mathcal{H}}$ as

\begin{align}{\label{operator-A}}
 \mathcal{A}u = \Delta^{2}u-2\Delta u,~ u \in D(\mathcal{A})
\end{align}

Some additional function spaces that we are going to use in this paper are

\begin{itemize}
    \item $C([0, T], \mathcal{V})$ with the norm $\sup_{0 \leq p \leq T} |u(p)|_\mathcal{H}$, for $u \in C([0, T], \mathcal{H})$.
    
    \item $L^2_w([0, T], D( \mathcal{A}))$ denotes the space $L^2([0, T], D( \mathcal{A}))$ endowed with the weak topology.
    
    \item $L^2([0, T],  \mathcal{V})$ with the norm $\|u\|_{L^2([0,T], \mathcal{V})} = \left(\int_0^T \|u(p)\|^2_\mathcal{V} \, dp\right)^{1/2}$.
    
    \item $C([0, T],  \mathcal{V}_w)$ contains $ \mathcal{V}_w$-valued continuous functions, in weak sense, on $[0, T]$.
\end{itemize}

\subsection{Tangent space and Orthogonal Projection}  

 The tangent space is given as 
 $$T_{u}{M}= \{~ h: ~~\langle h, u \rangle = 0~ ,~~\forall ~h \in \mathcal{H}\}$$
In addition, the map  $ \pi_{u}:\mathcal{H} \longrightarrow T_{u}{M}$ is the orthogonal projection onto $u$  and is given by: 
\begin{equation}{\label{lemma_Tangent}}
    \pi_{u}(h)= h-\langle h, u \rangle ~u, ~~~~~~h \in \mathcal{H}.
\end{equation}  

By considering $ u \in \mathcal{E} \cap {M}$ and applying the definition of orthogonal projection (\ref{lemma_Tangent}), the projection of $-\Delta^{2}u+2\Delta u -au - u^{2n-1} $ under the map $\pi_{u}$ using integration by parts  \cite{Brezis_2010} can be calculated as:\\
$\pi _{u}(-\Delta^{2}u+2\Delta u -au - u^{2n-1})$
\begin{align}{\label{prj}}
 &=-\Delta^{2}u+2\Delta u -au - u^{2n-1}
+\langle \Delta^{2}u-2\Delta u +au + u^{2n-1}, u \rangle ~u \notag \\
&=-\Delta^{2}u+2 \Delta u -au - u^{2n-1} + \langle \Delta^{2}u, u \rangle ~u -2\langle \Delta u, u \rangle ~u   ~+a\langle u, u  
\rangle ~u+\langle u^{2n-1}, u \rangle ~u \notag \\
&=-\Delta^{2}u+2 \Delta u -au - u^{2n-1} + \langle \Delta u,  \Delta u \rangle ~u -2\langle - \nabla u, \nabla u \rangle ~u   ~+a\langle u, u  
\rangle ~u+\langle u^{2n-1}, u \rangle ~u \notag \\
&=-\Delta^{2}u+2 \Delta u -au - u^{2n-1} + \| \Delta u\|^{2}_{{\mathcal{L}}^{2}(\mathcal{O})} ~u + 2\| \nabla u\|^{2}_{{\mathcal{L}}^{2}(\mathcal{O})} ~u   +au+\| u\|^{2n}_{{\mathcal{L}}^{2n}(\mathcal{O})} u \notag \\ 
&=-\Delta^{2}u+2 \Delta u  + \|  u\|^{2}_{{H}^{2}_{0}} ~u + 2\|  u\|^{2}_{{H}^{1}_{0}} ~u  ~ +\| u\|^{2n}_{{\mathcal{L}}^{2n}} u- u^{2n-1}
\end{align}

\begin{definition}\label{Martingale solution def}
There is a \textbf{martingale solution } of \eqref{main_eq_st_Ito}  if there exist
\begin{enumerate}
    \item a stochastic basis $(\hat{\Omega}, \hat{\mathcal{F}}, \hat{\mathbb{F}}, \hat{\mathbb{P}})$ with filtration $\hat{\mathbb{F}}$
    \item an $\mathbb{R}^d$-valued $\hat{\mathbb{F}}$-Wiener process $\hat{W}$
    \item and a $\hat{\mathbb{F}}$-progressively measurable process $u: [0, T] \times \hat{\Omega} \to D(\mathcal{A})$ with $\hat{\mathbb{P}}$ a.e. paths

    $$u(\cdot, \omega) \in C([0, T], \mathcal{V}_w) \cap L^2([0, T], D(\mathcal{A}))$$ such that for all $t \in [0, T]$ and $v \in \V$ we have
    \[
    \langle u(t), v \rangle - \langle u(0), v \rangle = \int_0^t \left\langle -\Delta^{2}u(p)+2 \Delta u(p)+F(u(p)) + \frac{1}{2}\sum_{k=1}^{N} m_{k}(u(p)), v\right\rangle dp + \sum_{k=1}^N \int_0^t \langle B_k(u), v \rangle dW_k,
    \]
    holds $\mathbb{P}$-a.s. \cite{ref2}
\end{enumerate}
\end{definition}

\begin{definition}[\textbf{Modulus of continuity}]\label{def:modulus_continuity}
Suppose $v \in C([0, T_1], S)$ then we define modulus of continuity of $v$ as
\[
m(v, \epsilon) = \sup_{\substack{p_1,p_2 \in [0,T_1] \\ |p_1-p_2| \leq \epsilon}} \rho(v(p_1), v(p_2))
\]
\end{definition}

\begin{theorem}\label{thm:sequence_condition}
\cite{ref2} The sequence $\{Y_k\}$ of random variables with values in $S$ satisfies \textbf{[T]} if and only if
\[
\forall \epsilon_1 > 0, \epsilon_2 > 0 \; \exists \delta > 0 : \sup_{k \in \mathbb{N}} \mathbb{P}\{m(Y_k, \delta) > \epsilon_2\} \leq \epsilon_1
\]
\end{theorem}

\begin{lemma}\label{lem:distribution_property}
Let $\{Y_k\}$ satisfy \textbf{[T]}. Assume $P_k$ be the distribution/law of $Y_k$ on $C([0, T], S)$, $k \in \mathbb{N}$. Then
\[
\forall \epsilon' > 0, \exists A_{\epsilon'} \subset C([0, T], S) : \sup_{k \in \mathbb{N}} P_k(A_{\epsilon'}) + \epsilon' \geq 1 \text{ and } \lim_{\delta \to 0} \sup_{v \in A_{\epsilon'}} (m(v, \delta)) = 0.
\]
\end{lemma}

\begin{definition}[\textbf{Aldous Condition}]\label{def:aldous_condition}
\cite{ref2} The sequence $(Y_k)_{k \in \mathbb{N}}$ satisfies Aldous condition \textbf{[A]} if and only if $\forall \epsilon_1 > 0, \epsilon_2 > 0, \exists \delta' > 0$ in such a way that for every sequence $\{s_k\}$ of stopping times with $s_k \leq T$ we have:
\[
\sup_{k \in \mathbb{N}} \sup_{0 \leq a \leq \delta} \mathbb{P}(\rho(Y_k(s_k + a), Y_k(s_k)) \geq \epsilon_2) \leq \epsilon_1
\]
\end{definition}

\begin{proposition}\label{prop:equivalence_conditions}
\textbf{[T]} and \textbf{[A]} are equivalent.
\end{proposition}

\begin{lemma}\label{lem:martingale_space}
The space consisting of the continuous and square-integrable processes, which are martingales and have values in $H$, is represented by $\mathcal{M}^2_T$. Moreover, it can be shown that $\mathcal{M}^2_T$ is a Banach space w.r.t following norm,
\[
\| u \|_{\mathcal{M}^2_T} = \sqrt{\mathbb{E}\left( \sup_{p \in [0,T]} \|u(p)\|^2_H \right)}.
\]
\end{lemma}

\subsection{Martingale Representation Theorem}

\begin{theorem}[Martingale Representation Theorem]\label{thm:martingale_representation}
Consider Hilbert space $\mathcal{H}$ and $M \in \mathcal{M}^2_T(H)$ with
\[
\langle M \rangle_s = \int_0^s \left( f(u)Q^{\frac{1}{2}} \right) \cdot \left( f(u)Q^{\frac{1}{2}} \right)^* du
\]
where $f(s)$ is a predictable process.

$\mathcal{H}_0 = Q^{\frac{1}{2}} \mathcal{H}$ is Hilbert space with following inner product
\[
\langle h, k \rangle_{\mathcal{H}_0} := \left\langle Q^{-\frac{1}{2}} h, Q^{-\frac{1}{2}} k \right\rangle_\mathcal{H}, \quad h, k \in \mathcal{H}_0.
\]
And $Q$ is a symmetric, bounded and non-negative operator in $\mathcal{H}$. Then one can find the probability space $(\hat{\Omega}, \hat{\mathcal{F}}, \hat{\mathbb{P}})$, a filtration $\{\mathcal{F}_s\}$ and a $\mathcal{H}$-valued $Q$-Wiener process $W$ defined on $(\Omega \times \hat{\Omega}, \mathcal{F} \times \hat{\mathcal{F}}, \mathbb{P} \times \hat{\mathbb{P}})$ adapted to $\{\mathcal{F} \times \hat{\mathcal{F}}\}$ such that for all $(\omega, \hat{\omega}) \in \Omega \times \hat{\Omega}$,
\[
M(s, \omega, \hat{\omega}) = \int_0^s f(u, \omega, \hat{\omega}) dW(u, \omega, \hat{\omega}), \quad s \in [0, T],
\]
with
\[
M(s, \omega, \hat{\omega}) = M(s, \omega), \text{ and } f(s, \omega, \hat{\omega}) = f(s, \omega).
\]
\end{theorem}

\section{\textbf{ \Large Faedo Galerkin Approximation}}

Assume that $\mathcal{H}$ has an orthonormal basis $\{e_{j}\}$ that contains the eigenvectors of $\mathcal{A}$ and $\{\lambda_{j}\}$ is the set of eigenvalues of $\mathcal{A}$. Suppose $\mathcal{H}_{n}$ of $\mathcal{H}$ such that $\mathcal{H}_n = \text{span}\{e_{j}\}_{j=1}^{n}$. 
Clearly, for $l\geq k$ we have $\mathcal{H}_{k} \subset \mathcal{H}_{l} $. The linear operator $ \mathcal{Z} _n:  \mathcal{H} \rightarrow \mathcal{H}_{n}$ can be expressed as follows:
\begin{eqnarray}
    \mathcal{Z} _{n}u := \sum_{j=1}^{n} \langle u , e_{j} \rangle e_{j}
 \end{eqnarray}

Therefore, the Faedo Galerkin Approximation for the equation (\ref{main_eq_st_Ito-2}) is

\begin{eqnarray}{\label{approx}}
    \begin{cases}
      du_n &= \left[\mathcal{Z} _{n}(-\Delta^{2}+2 \Delta) (u_n)+\mathcal{Z} _{n}(F(u_n))  \right] ~dt +  \sum_{k=1}^{N} \mathcal{Z}_n( B_{k}(u_n)) \circ dW_{k}  \\
      u_n(0)&=\frac{\mathcal{Z} _{n}(u_0)}{|\mathcal{Z} _{n}(u_0)|} 
      \end{cases} 
\end{eqnarray}

Since $e_j$ are eigen vectors with eigen values $\lambda_j$ of $-\Delta^2 + 2\Delta$ then
\begin{align*}
    (-\Delta^2 + 2\Delta)e_j= (-\lambda_j^2-2\lambda_j)e_j
\end{align*}

And  by definition of $\mathcal{Z}_n$, we get
\begin{align}\label{1}
\mathcal{Z}_n \left( (-\Delta^2 + 2\Delta) u_n \right) &= \sum_{k=1}^n \left\langle (-\Delta^2 + 2\Delta) u_n, e_k \right\rangle e_k 
= \sum_{k=1}^n \left\langle \sum_{j=1}^n c_j (-\lambda_j^2 - 2\lambda_j) e_j, e_k \right\rangle e_k \notag \\
&= \sum_{k=1}^n c_k (-\lambda_k^2 - 2\lambda_k) e_k = (-\Delta^2 + 2\Delta) u_n.
\end{align}

Also, 

\begin{align}\label{2}
\mathcal{Z}_n(B_j (u_n)) &= \mathcal{Z}_n(f_j - \langle f_j, u_n \rangle u_n) = \mathcal{Z}_n(f_j) - \langle f_j, u_n \rangle \mathcal{Z}_n(u_n) \notag \\
&= f_j - \langle f_j, u_n \rangle u_n = B_j (u_n)
\end{align}

Again,

\begin{align}\label{3}
    \mathcal{Z}_n(F(u_n)) &= \mathcal{Z}_n(\|  u_n\|^{2}_{{H}^{2}_{0}} ~u_n + 2\|    u_n\|^{2}_{{H}^{1}_{0}} ~u_n  +\| u_n\|^{2n}_{{\mathcal{L}}^{2n}} u_n- u_n^{2n-1}) \notag \\
    &=\|  u_n\|^{2}_{{H}^{2}_{0}} \mathcal{Z}_n(u_n)+ 2\|    u_n\|^{2}_{{H}^{1}_{0}} \mathcal{Z}_n(u_n)  +\| u_n\|^{2n}_{{\mathcal{L}}^{2n}} \mathcal{Z}_n(u_n)- \mathcal{Z}_n(u_n^{2n-1})\notag\\
    &= \|  u_n\|^{2}_{{H}^{2}_{0}} ~u_n + 2\|    u_n\|^{2}_{{H}^{1}_{0}} ~u_n  +\| u_n\|^{2n}_{{\mathcal{L}}^{2n}} u_n- u_n^{2n-1}\notag\\
    &=F(u_n)
\end{align}

Therefore, using (\ref{1}), (\ref{2}) and (\ref{3}), the equation (\ref{approx}) takes the form

\begin{eqnarray}{\label{approx_final}}
    \begin{cases}
      du_n &= \left[ -\Delta^{2}u_n+2 \Delta u_n +F(u_n)  \right] ~dt +  \sum_{k=1}^{N}  B_{k}(u_n) \circ dW_{k}  \\
      u_n(0)&=\frac{\mathcal{Z} _{n}(u_0)}{|\mathcal{Z} _{n}(u_0)|} 
      \end{cases} 
\end{eqnarray}

And its itô version can be written as

\begin{eqnarray}{\label{apprx}}
    \begin{cases}\label{apprx}
      du &= \left[-\Delta^{2}u+2 \Delta u+F(u) + \frac{1}{2}\sum_{k=1}^{N} m_{k}(u) \right] ~dt +  \sum_{k=1}^{N} B_{k}(u) dW_{k} \\ 
      u_n(0)&=\frac{\mathcal{Z} _{n}(u_0)}{|\mathcal{Z} _{n}(u_0)|} 
       \end{cases} 
\end{eqnarray}

\begin{theorem}\label{manifold}
    If $u_0 \in \mathcal{V}\cap  M$, then $u_n(t) $ that solves \eqref{apprx} is also in $M$.
\end{theorem}
\begin{proof}
     By applying the Itô lemma (see \cite{32}) to $ \gamma(u_n) = \frac{1}{2}|u_n|_{\mathcal{H}}^{2}$, we get \cite{SCMSHE}

\begin{align}{\label{Inv_eq}}
 &\left(|u_n|^{2}_{\mathcal{H}}-1\right) 
 = \sum_{k=1}^{N} \int_{0}^{t} 2\langle u_n(p), f_{k}\rangle(|u_n(p)|_{\mathcal{H}}^{2}-1) ~dW_{k}(p) \\
    &+ \int_{0}^{t \wedge \tau_{\ell}} \left( 2\|  u_n(p)\|^{2}_{\mathcal{H}^{2}_{0}}  + 4\|    u_n(p)\|^{2}_{\mathcal{H}^{1}_{0}}   +2\| u_n(p)\|^{2n}_{{\mathcal{L}}^{2n}} 
    +3 \langle f_{k}, u_n(p)\rangle^{2} - |f_{k}|^{2}\right)(|u_n(p)|_{\mathcal{H}}^{2}-1) ~dp, ~~\mathbb{P}-a.s
\end{align}

For more effortless simplification, set

\begin{align*}
    \eta ( t) &=  |u_n(t)|^{2}_{\mathcal{H}}-1\\
    a_{1}(t) & = 2\langle u_n(t), f_{1}\rangle \\
    a_{2}(t) &= \left( 2\|  u_n(t)\|^{2}_{\mathcal{H}^{2}_{0}}  + 4\|    u_n(t)\|^{2}_{\mathcal{H}^{1}_{0}}   +2\| u_n(t)\|^{2n}_{{\mathcal{L}}^{2n}} 
    +3 \langle f_{k}, u_n(t)\rangle^{2} - |f_{k}|^{2}\right) \\
    F_{1}(t,  \eta ( t)) &= a_{1}(t) \eta ( t) \\
    F_{2}(t, \eta ( t)) &= a_{2}(t) \eta ( t)
\end{align*}

We have

\begin{align}{\label{las_Pb}}
 \eta(t) 
 &=  \int_{0}^{t \wedge \tau_{\ell}} F_{1}(p,   \eta ( p)) ~dW_{k}(p) + \int_{0}^{t \wedge \tau_{\ell}} F_{2}(p,  \eta ( p))  ~dp, ~~~~~~~\mathbb{P}-a.s \notag\\
 \text{and}~~~~ \eta(0) &=  |u(0)|^{2}_{\mathcal{H}}-1 = 0
\end{align}

Referring \cite{SCMSHE}, we have a unique solution t equation \eqref{las_Pb}, that is, $\eta(0)= \eta (t)$ and
\begin{equation*}
 |u_n(t)|^{2}_{\mathcal{H}}= 1. ~~~~~ \forall t \in [0,T]
\end{equation*}
 Thus, this proves other theorem.
\end{proof}

\section{\textbf{\Large Tightness of Measures}}
This section is devoted to the proof that the is set of measures $\{\mathcal{L}(u_n) : n \in \mathbb{N}\}$ is tight on $\mathcal{X}_{T}$. Before showing this, we prove some important estimates.
\begin{theorem}\label{thm:4.1}
Let $u_n(t)$ solves \eqref{approx} then
\begin{align}
&\sup_{ 1\leq n} \mathbb{E}\left[ \sup_{p\in[0,T]} \|u_n(p)\|_{\mathcal{V}}^2 \right] \leq K_1, \label{eq:5.1} \\
&\sup_{1\leq n} \mathbb{E}\left[ \sup_{p\in[0,T]} \|u_n(p)\|^{2n}_{\mathcal{L}^{2n}} \right] \leq K_2, \label{eq:5.2} \\
&\sup_{1\leq n} \mathbb{E}\left[ \int_0^T \|u_n(p)\|^2_{\mathcal{E}} \, dp \right] \leq K_3, \label{eq:5.3}
\end{align}
for some constants $K_1, K_2, K_3 > 0$.
\end{theorem}

\begin{proof}
   In \cite{SCMSHE} we have prove that, for the equation $\mathcal{Y}(u) = \frac{1}{2} \|u\|^{2}_{\mathcal{V}} + \frac{1}{2n} \|u\|^{2n}_{\mathcal{\mathcal{L}}^{2n}}$, expectation $E\{\mathcal{Y}(u)\}$ is uniformly bounded. It follows that there are constant $K_1$ and $K_{2}$ such that

\begin{align*}
&\sup_{ 1\leq n} \mathbb{E}\left[ \sup_{p\in[0,T]} \|u_n(p)\|_{\mathcal{V}}^2 \right] \leq K_1 \\
&\sup_{1\leq n} \mathbb{E}\left[ \sup_{p\in[0,T]} \|u_n(p)\|^{2n}_{\mathcal{L}^{2n}} \right] \leq K_2
\end{align*}

Now, let us prove \eqref{eq:5.3}. To prove this inequality, we will apply the Itô lemma to the following equation.

\begin{align}{\label{psi}}
    \Psi (u) = \frac{1}{2}\|u_n\|^2_{\mathcal{V}}
\end{align}

From \cite{JH}, we have

\begin{align}\label{eqs}
    d_{u_n}(\Psi (h) ) = \langle u_n, h\rangle_{\mathcal{V}},~~and \qquad   d^2_{u}(\Psi (h,h') ) = \langle h, h'\rangle_{\mathcal{V}}
\end{align}

using Itö lemma together with \eqref{eqs} to \eqref{psi}, we get

\begin{align}{\label{Ito-lemma-on energy}}
    \Psi \left(u_n(t )\right) -\Psi(u_n({0}))&= \sum_{k=1}^{N} \int_{0}^{t } \langle \Psi'(u_n(p)) , B_{k}(u_n(p)) \rangle ~dW_{k}(p)+\frac{1}{2} \sum_{k=1}^{N} \int_{0}^{t} \langle \Psi'(u_n(p)) , m_{k}(u_n(p)) \rangle ~dp \notag \\
    &~~~~+\frac{1}{2} \sum_{k=1}^{N} \int_{0}^{t } \Psi''(u_n(p))\left( B_{k}(u_n(p)) , B_{k}(u_n(p)) \right) ~dp \notag \\
    &~~~~+ \int_{0}^{t } \langle \Psi'(u_n(p)) , -\Delta^{2}u_n(p)+2 \Delta u_n(p)+F(u_n(p)) \rangle ~dp \notag\\
    &= \sum_{k=1}^{N} I_{1,k} + \sum_{k=1}^{N} I_{2,k} + \sum_{k=1}^{N} I_{3,k} + I_{4} , ~~~~~~~\mathbb{P}-a.s. ~~ \forall t \in [0,T]
\end{align}

Now, we will deal each integral separately. Let us start with $I_{1,k}$.\\
In \cite{SCMSHE}, we shown that the integral $I_{1,k}$ is martingale and

\begin{align}{\label{I-1,kintegral}}
    \mathbb{E} \left(I_{1,k}\right) =0
\end{align}

Estimate $I_{2,k}$.

\begin{align}\label{1}
   \int_{0}^{t} \langle \Psi'(u_n(p)), m_{k}(u_n(p))\rangle dp&=  \int_{0}^{t} \left \langle u_n(p),m_{k}(u_n(p))\right \rangle _{\mathcal{V}}~dp = \int_0^t \left \langle u_n, - \left\langle f_{k}, B_{k}(u_n) \right\rangle u_n- \left\langle f_{k}, u_n \right\rangle B_{k}(u_n)\right \rangle _{\mathcal{V}} ~dp \notag\\
   &=\int_0^t \left(-\|u_n\|_{\mathcal{V}}^{2} \left\langle f_{k}, B_{k}(u_n(p))\right\rangle- \left\langle f_{k}, u_n(p)\right\rangle \left\langle u_n, B_{k}(u_n(p))\right\rangle_{\mathcal{V}} \right) ~dp \notag\\
&\leq \left|\int_0^t -\lVert u_n(p) \rVert^2_{\mathcal{V}} \langle f_k, B_k(u_n) \rangle_{\mathcal{H}} - \langle f_k, u_n(p) \rangle_{\mathcal{H}} \langle u_n(p), B_k(u_n(p)) \rangle_{\mathcal{V}} \, dp \right| \notag\\
&\leq \int_0^t \left| -\lVert u_n(p) \rVert^2_{\mathcal{V}} \langle f_k, B_k(u_n) \rangle_{\mathcal{H}} - \langle f_k, u_n(p) \rangle_{\mathcal{H}} \langle u_n(p), B_k(u_n(p)) \rangle_{\mathcal{V}} \right| \, dp \notag\\
&\leq \int_0^t \left| \lVert u_n(p) \rVert^2_{\mathcal{V}} \langle f_k, B_k(u_n) \rangle_{\mathcal{H}} \right| + \left| \langle f_k, u_n(p) \rangle_{\mathcal{H}} \langle u_n(p), B_k(u_n(p)) \rangle_{\mathcal{V}} \right| \, dp \notag\\
&\leq \int_0^t \lVert u_n(p) \rVert^2_{\mathcal{V}} \left| \langle f_k, B_k(u_n) \rangle_{\mathcal{H}} \right| + \left| \langle f_k, u_n(p) \rangle_{\mathcal{H}} \right| \left| \langle u_n(p), B_k(u_n(p)) \rangle_{\mathcal{V}} \right| \, dp \notag\\
&\leq \int_0^t \lVert u_n(p) \rVert^2_{\mathcal{V}} \lVert f_k \rVert^2_{\mathcal{H}} \lVert B_k(u_n) \rVert^2_{\mathcal{H}} + \lVert f_k \rVert^2_{\mathcal{H}} \lVert u_n(p) \rVert^2_{\mathcal{H}} \left| \langle u_n(p), B_k(u_n(p)) \rangle_{\mathcal{V}} \right| \, dp \notag\\
&\leq \int_0^t \lVert u_n(p) \rVert^2_{\mathcal{V}} \lVert f_k \rVert^2_{\mathcal{H}} \lVert B_k(u_n) \rVert^2_{\mathcal{H}} + \lVert f_k \rVert^2_{\mathcal{H}} \lVert u_n(p) \rVert^2_{\mathcal{H}} \lVert u_n(p) \rVert^2_{\mathcal{V}} \lVert B_k(u_n(p)) \rVert^2_{\mathcal{V}} \, dp 
\end{align}

But
\begin{align}\label{2}
\lVert B_k(u_n) \rVert^2_\mathcal{V}  &= \lVert f - \langle f_k, u_n \rangle_{\mathcal{H}} u_n \rVert^2_\mathcal{V}  \notag \\
\lVert B_k(u_n) \rVert^2_\mathcal{V}  &\leq \lVert f_k \rVert^2_\mathcal{V}  + |f_k|^2_{\mathcal{H}} |u_n|^2_{\mathcal{H}} \lVert u_n \rVert^2_\mathcal{V} 
\end{align}

And
\begin{equation}\label{3}
f_k \in \mathcal{V} \hookrightarrow \mathcal{H},
\end{equation}
Also, we know that
\begin{equation}\label{4}
\lVert u_n \rVert_{\mathcal{V}}^2 < C_1.
\end{equation}

Combining all above inequalities \eqref{1},\eqref{2},\eqref{3}, and \eqref{4}, we can find a constant $C_2$ such that

\begin{align}\label{I-2,kintegral}
    I_{2,k} \leq C_{2}
\end{align}

Again consider the integral $I_{3,k}$.

\begin{align}
   \int_{0}^{t } \Psi''(u_n(p))\left( B_{k}(u_n(p)) , B_{k}(u_n(p)) \right) ~dp &=    \int_{0}^{t } \|B_{k}(u_n(p))\|_{\mathcal{V}}^{2}~dp
\end{align}

Using \eqref{2},\eqref{3}, and \eqref{4} we can deduce that there is a constant $C_3$ such that

\begin{align}\label{I-3,kintegral}
    I_{3,k} \leq C_{3}
\end{align}

Finally, we will find the estimation for $I_{4,k}$.\\
$\int_{0}^{t } \langle \Psi'(u_n(p)) , -\Delta^{2}u_n(p)+2 \Delta u_n(p)+F(u_n(p)) \rangle ~dp$

\begin{align*}
     &= \int_{0}^{t } \langle u_n(p) , -\Delta^{2}u_n(p)+2 \Delta u_n(p)+F(u_n(p)) \rangle_{\mathcal{V}} ~dp\notag \\
    &= \int_{0}^{t }\left( \langle u_n(p) , -\Delta^{2}u_n(p)+2 \Delta u_n(p) \rangle_{\mathcal{V}} + \langle u_n(p) , F(u_n(p)) \rangle_{\mathcal{V}}\right)~dp\notag \\
    &\leq \left|\int_{0}^{t } \langle u_n(p) , -\Delta^{2}u_n(p)+2 \Delta u_n(p)+F(u_n(p)) \rangle_{\mathcal{V}} ~dp\right|\notag \\
    & \leq \int_{0}^{t }  \left(\|u_n(p)\|_\mathcal{V}  \| -\Delta^{2}u_n(p)+2 \Delta u_n(p)\|_\mathcal{V}+ \|\langle u_n(p) , F(u_n(p)) \rangle_{\mathcal{V}}\| \right)~dp\notag \\
    & \leq \int_{0}^{t }  \left(\|u_n(p)\|_\mathcal{V}  \| \mathcal{A}u_n(p)\|_\mathcal{V}+ \|\langle u_n(p) , F(u_n(p)) \rangle_{\mathcal{V}}\| \right)~dp\notag \\
\end{align*}

But, from elliptic regularity for 
$\mathcal{A}$ (with clamped boundary conditions) \cite{Mazya}, we can infer that

\begin{align}
    c_1 \|Au_n\|_{\mathcal{V}} \leq \|u_n\|_{D(A)} \leq c_2 \|Au_n\|_{\mathcal{V}}
\end{align}
where \( \|u_n\|_{D(A)} = \left( \|u_n\|_{\mathcal{V}}^2 + \|Au_n\|_{\mathcal{V}}^2 \right)^{1/2} \).

Therefore, there is a constant $C_2$ such that

\begin{align}{\label{Au}}
      \|Au_n\|_{\mathcal{V}} \leq -C_{2} \|u_n\|_{D(A)}
\end{align}

But,

\begin{align}\label{4.16}
    \| \langle u_n(p) , F(u_n(p)) \rangle_{\mathcal{V}} \| \leq  \| u_n(p)  \|_{\mathcal{V}}   \| F(u_n(p))  \|_{\mathcal{V} }
\end{align}

Substituting $u_{1}=u_n$, and $u_{2}$ in the lemma 2.8 of \cite{SCMSHE} and using $\mathcal{V} \hookrightarrow  \mathcal{H}$, we have 

\begin{align}\label{4.17}
    \| F(u_n(p))  \|_{\mathcal{V}} \leq \left( 2C \|u_n(p)_{}\|_{\mathcal{V}}^{2} \notag +C_{n} \left[  \left( \frac{2n-1}{2}\right)  \|u_n(p)_{}\|_{\mathcal{V}}^{2n-1}\|u_n(p)_{}\|_{\mathcal{V}}  +  \|u_n(p)_{}\|_{\mathcal{V}}^{2n} + \left( 1+ \|u_n(p)_{}\|_{\mathcal{V}}^{2}\right)^{\frac{1}{3}} \right] \right)\|u_n(p)_{}\|_{\mathcal{V}} \tag{4.17}
\end{align}

Using \eqref{4.16}, and \eqref{4.17} together with AM-GM inequality $\|u_n(p)\|_\mathcal{V} \cdot \|u_n(p)\|_{D(\mathcal{A})} \leq \frac{\|u_n(p)\|_\mathcal{V}^2 + \|u_n(p)\|_{D(\mathcal{A})}^2}{2}$, we have

$$\int_{0}^{t } \langle \Psi'(u_n(p)) , -\Delta^{2}u_n(p)+2 \Delta u_n(p)+F(u_n(p)) \rangle ~dp \leq   $$
\begin{align*}
  \int_{0}^{t }  
    \begin{bmatrix}
       -\frac{C_{2}}{2}\left({\|u_n(p)\|_\mathcal{V}^2 + \|u_n(p)\|_{D(\mathcal{A})}^2}\right)+  2C \|u_n(p)_{}\|_{\mathcal{V}}^{4} \notag \\+C_{n} \left\{  \left( \frac{2n-1}{2}\right) \left( \|u_n(p)_{}\|_{\mathcal{V}}^{2n-1}\right) (\|u_n(p)_{}\|_{\mathcal{V}})  + \left( \|u_n(p)_{}\|_{\mathcal{V}}^{2n}\right) + \left( 1+ \|u_n(p)_{}\|_{\mathcal{V}}^{2}\right)^{\frac{1}{3}} \right\} \|u_n(p)_{}\|^{2}_{\mathcal{V}}
    \end{bmatrix} ~dp\notag \\
\end{align*}

$$\int_{0}^{t } \langle \Psi'(u_n(p)) , -\Delta^{2}u_n(p)+2 \Delta u_n(p)+F(u_n(p)) \rangle ~dp \leq$$
\begin{align*}
    -\frac{C_{2}C_{3}}{2}T  \int_{0}^{t }  \| u_n(p)\|^{2}_{D(\mathcal{A})}~dp ~~ + 2CC^{4}_{3}T^{4}+ C_{n} \left\{  \left( \frac{2n-1}{2}\right) \left( C_{3}^{2n}T^{2n}\right)   + \left( C_{3}^{2n}T^{2n}\right) + \left( 1+ C_{3}^{2}T^{2}\right)^{\frac{1}{3}} \right\} 
\end{align*}

It follow that 

\begin{align} \label{I-4,kintegral}
    I_{4,k} \leq -C'_{3} \int_{0}^{t }  \| u_n(p)\|^{2}_{D(\mathcal{A})}~dp ~~ + C'_{4}
\end{align}

Where $C'_{3}=  \frac{C_{2}C_{3}}{2}T $, $C'_{4}=2CC^{4}_{3}T^{4}+ C_{n} \left\{  \left( \frac{2n-1}{2}\right) \left( C_{3}^{2n}T^{2n}\right)   + \left( C_{3}^{2n}T^{2n}\right) + \left( 1+ C_{3}^{2}T^{2}\right)^{\frac{1}{3}} \right\} $

Now, apply the expectation on equation \eqref{Ito-lemma-on energy} to the both sides. Also using \eqref{I-1,kintegral}, \eqref{I-2,kintegral},\eqref{I-3,kintegral} and \eqref{I-4,kintegral}, it follows that

\begin{align*}
    \mathbb{E}\left(   \Psi \left(u_n(t )\right) -\Psi(u_n({0}))\right)& =  \mathbb{E}\left(\sum_{k=1}^{N} I_{1,k} + \sum_{k=1}^{N} I_{2,k}+ \sum_{k=1}^{N} I_{3,k} + I_{4}\right) \\
   \frac{1}{2}\|u_n\|^2_{\mathcal{V}} -\frac{1}{2}\|u_{n}(0)\|^2_{\mathcal{V}} &\leq   \sum_{k=1}^{N} C_2+ \sum_{k=1}^{N} C_3 -C'_{3} \mathbb{E}\left(\int_{0}^{t }  \| u_n(p)\|^{2}_{D(\mathcal{A})}~dp\right) ~~ + C'_{4}\\
   \mathbb{E}\left(\int_{0}^{t }  \| u_n(p)\|^{2}_{D(\mathcal{A})}~dp\right) &\leq \frac{1}{C'_3}\left(\sum_{k=1}^{N} C_2+ \sum_{k=1}^{N} C_3+C'_4+ \frac{1}{2}\|u_{n}(0)\|^2_{\mathcal{V}}-\frac{1}{2}C_{5}\right)
\end{align*}

Thus, by taking $K_{3}=\frac{1}{C'_3}\left(\sum_{k=1}^{N} C_2+ \sum_{k=1}^{N} C_3+C'_4+ \frac{1}{2}\|u_{n}(0)\|^2_{\mathcal{V}}-\frac{1}{2}C_{5}\right)$, we get

\begin{align}
    \mathbb{E}\left(\int_{0}^{t }  \| u_n(p)\|^{2}_{\mathcal{E}}~dp\right) \leq K_3
\end{align}
We are done with the proof.
\end{proof}

\begin{theorem}\label{tight thrm}
Suppose that $\{\mathcal{L}(u_n) : n \in \mathbb{N}\}$ is a set of measures, then this set is tight on $(\mathcal{X}_T, \mathcal{F})$. Where
$$\mathcal{X}_T = C([0, T], {L}^2;\mathcal{H}) \cup L^2_w([0, T]; D(\mathcal{A})) \cup L^2([0, T]; \mathcal{V}) \cup C([0, T], \mathcal{V}_w)$$
    
\end{theorem}
\begin{proof}

To prove that the given sequence $\{\Lambda(u_n) : n \in \mathbb{N}\}$ is tight on $(\mathcal{X}_T, \mathcal{F})$, it is sufficient for us to show that it obeys the Aldous conditions \textbf{[A]} \eqref{def:aldous_condition}. Let $(\tau_{n})_{n \in \mathbb{N}}$ be the sequence of stopping times such that $0 \leq \tau_{n} \leq T$. From the equation \eqref{apprx}, we have

\begin{align}
      u_n(t)&= u_n(0)+\int_0^t \left( -\Delta^{2}u_n(p)+2 \Delta u_n(p)\right) dp +\int_0^t F(u_n(p))~dp + \frac{1}{2}\sum_{k=1}^{N} \int_0^t m_{k}(u_n(p))~ dp + \sum_{k=1}^N \int_0^t B_k(u_n(p)) dW_k \notag\\
      &= \mathcal{J}^{n}_{0}+ \mathcal{J}^{n}_{1}+\mathcal{J}^{n}_{2}+ \frac{1}{2}\sum_{k=1}^{N}\mathcal{J}^{n}_{3}+\sum_{k=1}^{N}\mathcal{J}^{n}_{4}
\end{align}

 Let us estimate $\mathcal{J}^{n}_{i}$ for $i=0, 1,2,3,4$. Let $\lambda>0$. And $\mathcal{A}= \Delta^{2}-2 \Delta: \mathcal{E} \rightarrow\mathcal{H}$  is linear and continuous so applying Holder inequality together with \eqref{Au}  and \eqref{eq:5.3} we get

\begin{align}\label{5.2}
    \mathbb{E}\left(\left|\mathcal{J}^{n}_{1}(\tau_{n}+\lambda)-\mathcal{J}^{n}_{1}(\tau_{n})\right|_{\mathcal{H}}\right)&=  \mathbb{E}\left(\left|\int_{\tau_{n}}^{\tau_{n}+\lambda}\mathcal{A}u_n(p)~dp\right|_{\mathcal{H}}\right)\notag \leq \mathbb{E}\left(\int_{\tau_{n}}^{\tau_{n}+\lambda}\left|\mathcal{A}u_n(p)\right|_{\mathcal{H}}~dp\right)\notag \\
    &\leq C_6\mathbb{E}\left(\int_{\tau_{n}}^{\tau_{n}+\lambda}\left\|u_n(p)\right\|_{\mathcal{E}}~dp\right)\notag \leq C_6\lambda^{\frac{1}{2}}\left(\mathbb{E}\left(\int_{0}^{T}\left\|u_n(p)\right\|^{2}_{\mathcal{E}}~dp\right)\right)^{\frac{1}{2}}\notag \\
&\leq  C_6\lambda^{\frac{1}{2}}K_{3}^{\frac{1}{2}}
\end{align}

Consider $\mathcal{J}^{n}_{2}$

\begin{align}\label{EL2}
    \mathbb{E}\left(\left|\mathcal{J}^{n}_{2}(\tau_{n}+\lambda)-\mathcal{J}^{n}_{2}(\tau_{n})\right|_{\mathcal{H}}\right)&=  \mathbb{E}\left(\left|\int_{\tau_{n}}^{\tau_{n}+\lambda}F(u_n(p))~dp\right|_{\mathcal{H}}\right) \leq \mathbb{E}\left(\int_{\tau_{n}}^{\tau_{n}+\lambda}\left|F(u_n(p))\right|_{\mathcal{H}}~dp\right)
\end{align}

But, in \cite{Deterministic JS}, we have estimated $F(u)$ as
\begin{equation*}
    \|F(u_{1})-F(u_{2})\|_{\mathcal{H}} \leq \mathcal{G} ( \|u_{1}\|_{\mathcal{V}},\|u_{2}\|_{\mathcal{V}}) \|u_{1}-u_{2}\|_{\mathcal{V}}, \quad u_{1}, u_{2}\in \mathcal{V}
\end{equation*} 

where
\begin{align*}
    \mathcal{G} \left(\|u_1\|, \|u_2\|\right) &= 2C \left(\|u_1\|^{2} + \|u_2\|^{2} + \|u_1\|\|u_2\|\right)  + C_{n}  \begin{bmatrix}\left( \frac{2n - 1}{2} \right) \left( \|u_1\|^{2n - 1} + \|u_2\|^{n - 1} \right) (\|u_1\| + \|u_2\|)  \\+ \left( \|u_1\|^{2n} + \|u_2\|^{2n} \right) + \left( 1 + \|u_1\|^{2} + \|u_2\|^{2} \right)^{\frac{1}{3}} \end{bmatrix} 
\end{align*}

Substituting $u_{1}=u_{n}$ and $u_2=0$ we have

\begin{align}
    |F(u_{n})|_{\mathcal{H}} & \leq  \mathcal{G} \left(\|u_n\|, 0\right)\|u_n\|_{\mathcal{V}} \label{5.3}, \qquad \text{and} \\
     \mathcal{G} \left(\|u_n\|, 0\right) &\leq 2C \|u_n\|_{\mathcal{V}}^{2}   + C_{n}  \begin{bmatrix}\left( \frac{2n - 1}{2} \right) \|u_n\|_{\mathcal{V}}^{2n }    +  \|u_n\|_{\mathcal{V}}^{2n} + \left( 1 + \|u_n\|_{\mathcal{V}}^{2} \right)^{\frac{1}{3}} \end{bmatrix}  \|u_{n}\|_{\mathcal{V}}
    \end{align} 

From theorem \eqref{thm:4.1}, we can find a constant $C_7$ such that

\begin{align}\label{5.5}
    \mathcal{G} \left(\|u_n\|, 0\right)\leq C_7
\end{align}

Now, using \eqref{5.3} and \eqref{5.5}, we can reduce the inequality \eqref{EL2} as
\begin{align}\label{5.7}
    \mathbb{E}\left(\left|\mathcal{J}^{n}_{2}(\tau_{n}+\lambda)-\mathcal{J}^{n}_{2}(\tau_{n})\right|_{\mathcal{H}}\right) &\leq \mathbb{E}\left(\int_{\tau_{n}}^{\tau_{n}+\lambda}C_7\left\|u_n(p)\right\|_{\mathcal{V}}~dp\right) \leq \mathbb{E}\left(\int_{0}^{T}C_7\left\|u_n(p)\right\|_{\mathcal{V}}~dp\right)\notag \\
    &\leq \mathbb{E}\left[\left(\int_{0}^{T}C^{2}_7~dp\right)^{\frac{1}{2}}\left(\int_{0}^{T}\left\|u_n(p)\right\|^{2}_{\mathcal{V}}~dp\right)^{\frac{1}{2}}\right]\notag \\
    &\leq C^{}_7 T^{\frac{1}{2}} \left(\mathbb{E}\left(\int_{0}^{T}\left\|u_n(p)\right\|^{2}_{\mathcal{V}}~dp\right)\right)^{\frac{1}{2}} \leq C^{}_7 T^{\frac{1}{2}} K_{1}^{\frac{1}{2}}
\end{align}

Again, consider $\mathbb{E}(|{\mathcal{J}}^n_{4}(\tau_n + \lambda) - \mathcal{J}^n_3(\tau_n)|_{\mathcal{H}})$
\begin{align}\label{5.8}
\E\left(|\mathcal{J}^n_3(\tau_n + \lambda) - \mathcal{J}^n_3(\tau_n)|_{\mathcal{H}}\right) &= \E\left(\left|\int_{\tau_n}^{\tau_n+\lambda} m_k(u_n(p))dp\right|_{\mathcal{H}}\right) \notag \\
&\leq \E\left(\int_{\tau_n}^{\tau_n+\lambda} |m_k(u_n(p))|_{\mathcal{H}} dp\right) 
\end{align}

By referring \cite{Hussain_2005}, we have
\begin{align}\label{mk}
\|m_k(u_n) - m_k(v_n)\|_{\mathcal{V}} \leq C\|f_k\|^2_{\mathcal{V}} \left[2 + \|u_n\|_{\mathcal{V}}^2 + \|v_n\|_{\mathcal{V}}^2 + (\|u_n\|_{\mathcal{V}} + \|v_n\|_{\mathcal{V}})^2\right] \|u_n - v_n\|_{\mathcal{V}}
\end{align}

Inequality \eqref{mk} together $v_n = 0$ and Young's Inequality yields
\begin{align*}
\|m_k(u_n)\|_{\mathcal{V}} &\leq C\|f_k\|^2_{\mathcal{V}} \left(2 + \|u_n\|_{\mathcal{V}}^2 + \|u_n\|_{\mathcal{V}}^2\right) \|u_n\|_{\mathcal{V}} \\
&= C\|f_k\|^2_{\mathcal{V}} \left(2\|u_n\|_{\mathcal{V}} + 2\|u_n\|_{\mathcal{V}}^3\right) \\
&= 2C\|f_k\|^2_{\mathcal{V}} \|u_n\|_{\mathcal{V}} + 2C\|f_k\|^2_{\mathcal{V}} \|u_n\|_{\mathcal{V}}^3 \\
&\leq \frac{4C^2\|f_k\|_{\mathcal{V}}^4}{2} + \frac{\|u_n\|_{\mathcal{V}}^2}{2} + \frac{4C^2\|f_k\|_{\mathcal{V}}^4}{2} + \frac{(\|u_n\|_{\mathcal{V}}^2)^3}{2} \\
&\leq 4C^2\|f_k\|_{\mathcal{V}}^4 + \frac{\|u_n\|_{\mathcal{V}}^2}{2} + \frac{(\|u_n\|_{\mathcal{V}}^2)^3}{2},
\end{align*}
 Using Theorem \eqref{thm:4.1}, $\mathcal{V} \hookrightarrow \mathcal{H}$ and $\|f_k\|_{\mathcal{V}} < \infty$ imply that there is a constant $C_8$ such that

\begin{equation}\label{m_K}
|m_k(u)|_{\mathcal{H}} \leq C_{8}
\end{equation}

Using \eqref{m_K}, and \eqref{5.8}, we can infer that

\begin{equation}\label{5.11}
\E \left(|\mathcal{J}^n_3(\tau_n + \lambda) - \mathcal{J}^n_3(\tau_n)|_{\mathcal{H}}\right) \leq \lambda C_{8}
\end{equation}

Finally, Let us estimate  $\E\left(|\mathcal{J}^n_4(\tau_n + \lambda) - \mathcal{J}^4_n(\tau_n)|^2_{\mathcal{H}}\right)$. With the application of It\^o isometry, we can deduce that;
\begin{align}\label{5.12}
\E\left(|\mathcal{J}^n_4(\tau_n + \lambda) - \mathcal{J}^n_4(\tau_n)|^2_{\mathcal{H}}\right) &= \E\left(\left|\int_{\tau_n}^{\tau_n+\lambda} B_k(u_n(p))dW_k\right|_{\mathcal{H}}^2\right) \notag\\
&\leq \E\left(\int_{\tau_n}^{\tau_n+\lambda} |B_k(u_n(p))|_{\mathcal{H}}^2 dp\right)
\end{align}

However in \cite{Hussain_2005},
\begin{equation*}
\|B_k(u_n) - B_k(v_n)\|_{\mathcal{V}} \leq \|f_k\|_{\mathcal{V}}(\|u_n\|_{\mathcal{V}} + \|v_n\|_{\mathcal{V}})\|u_n - v_n\|_{\mathcal{V}}.
\end{equation*}
Setting $v = 0$ and using theorem \eqref{thm:4.1}, it follows that
\begin{align}\label{5.13}
\|B_k(u_n) - B_k(0)\|_{\mathcal{V}} &\leq \|f_k\|_{\mathcal{V}}\|u_n\|_{\mathcal{V}}^2 \notag \\
\|B_k(u_n) -f_k\|_{\mathcal{V}} &\leq \|f_k\|_{\mathcal{V}}\|u_n\|_{\mathcal{V}}^2 \leq \|f_k\|_{\mathcal{V}}C_1
\end{align}

But $\mathcal{V} \hookrightarrow \mathcal{H}$, so
\begin{align}\label{5.14}
|B_k(u_n)|_{\mathcal{H}} &\leq C \|B_k(u_n)\|_{\mathcal{V}}, \notag\\
&= C \|B_k(u_n) - f_k + f_k\|_{\mathcal{V}}, \notag\\
&\leq C \|B_k(u_n) - f_k\|_{\mathcal{V}} + C \|f_k\|_{\mathcal{V}}
\end{align}

Using inequalities \eqref{5.13} and \eqref{5.14} with $\|f_k\|_{\mathcal{V}} < \infty$ it follows that,
\begin{equation}\label{5.15}
|B_k(u_n)|_{\mathcal{H}} \leq CC_1\|f_k\|_{\mathcal{V}} + C\|f_k\|_{\mathcal{V}} \coloneqq C_{9}
\end{equation}

Merging \eqref{5.13} in \eqref{5.15} yields
\begin{equation}{\label{5.16}}
\E\left(|\mathcal{J}^n_4(\tau_n + \lambda) - \mathcal{J}^n_4(\tau_n)|^2_{\mathcal{H}}\right) \leq \E\left(\int_{\tau_n}^{\tau_n+\lambda} C_{9}^2 dp\right) = C_{9}^2 \lambda
\end{equation}

Fix $\kappa, \epsilon > 0$. Then using Chebyshev's inequality together with \eqref{5.2}, \eqref{5.7}, \eqref{5.11} and \eqref{5.16}, results in
\begin{equation*}
P\left(|\mathcal{J}^n_i(\tau_n + \lambda) - \mathcal{J}^n_i(\tau_n)|_{\mathcal{H}} \geq \kappa\right) \leq \frac{1}{\kappa} \E \left[|\mathcal{J}^n_i(\tau_n + \lambda) - \mathcal{J}^n_i(\tau_n)|_{\mathcal{H}}\right] \leq \frac{c_i \lambda}{\kappa},
\end{equation*}
with $n \in \N$, $i = 0,1, 2, 3$. Let $\delta_i = \frac{\kappa}{c_i} \epsilon$. Then
\begin{equation*}
\sup_{n \in \N} \sup_{0 \leq \lambda \leq \delta_i} P\left(|\mathcal{J}^n_i(\tau_n + \lambda) - \mathcal{J}^n_i(\tau_n)|_{\mathcal{H}} \geq \kappa\right) \leq \epsilon,
\end{equation*}

Again using Chebyshev's inequality, we get
\begin{equation*}
P\left(|\mathcal{J}^n_4(\tau_n + \lambda) - \mathcal{J}^n_4(\tau_n)|_{\mathcal{H}} \right) \leq \frac{1}{\kappa^2} \E \left[|\mathcal{J}^n_4(\tau_n + \lambda) - \mathcal{J}^n_4(\tau_n)|_{\mathcal{H}}^{2}\right] \leq \frac{c_4 \lambda}{\kappa^2}
\end{equation*}

Thus, choose $\delta_{4}= \frac{\kappa^2}{c_{4}}\epsilon$ and 
\begin{equation*}
\sup_{n \in \N} \sup_{0 \leq \lambda \leq \delta_4} P \left(|\mathcal{J}^n_4(\tau_n + \lambda) - \mathcal{J}^n_4(\tau_n)|_{\mathcal{H}} \geq \kappa\right) \leq \epsilon,
\end{equation*}

Since each of $\mathcal{J}^i_n$ satifies the Aldous condition, so we can say that it is also true for $u_n$. Hence, we are done with the  proof.
\end{proof}

\section{\textbf{\Large Convergence of Quadratic Variations}}

In this section, we are going to state some convergence results that are necessary to apply the martingale representation theorem. So, from theorem \eqref{tight thrm}  we have $(u_{n_{i}})_{i\in\NN}$ with probability measure space $(\hat{\Omega}, \FF, \hat{\PP})$ in a manner that we have $\mathcal{X}_T-$valued random variables $\widehat{u}$ and $\widehat{u}_{n_{i}}$, $i \geq 1$ with property that $\widehat{u}_{n_{i}}$ and $u_{n_{i}}$ have same distribution and $\widehat{u} _{n_{i}}$ converge to $\widehat{u}$ in $\mathcal{X}_T$ $\hat{\PP}$-a.s. That is,
\begin{equation}\label{same distribution}
\begin{cases}
 \widehat{u}_{n_{i}} \to \widehat{u} \quad \text{ in } \quad\CC([0, T],L^2), \\ \widehat{u}_{n_{i}} \to \widehat{u} \quad \text{ in } \quad L^2([0, T], \DD(\mathcal{A}))\\ \widehat{u}_{n_{i}} \to \widehat{u} \quad\text{ in }\quad  L^2([0, T], \VV), \\
 \widehat{u}_{n_{i}} \to \widehat{u} \quad \text{ in } \quad\CC([0, T], \VV_w)
    \end{cases}
\end{equation}
We denote $\widehat{u}_{n_{i}}$ again as $\widehat{u}_n$ because 
\begin{enumerate}
 \item $\widehat{u}_n \in $  $\CC([0, T], \H_n)$ $\PP$-a.s. \quad for each $n \in \NN$
    \item $\CC([0, T], \mathcal{H}_n) \subset \CC([0, T], \mathcal{H}) \cup L^2([0, T]; \VV)$ is Borel subset.
    \item $\widehat{u}_n$ and $u_n$ have same distribution.
\end{enumerate}
Therefore, for all $n \geq 1$,
\[
\LL(\widehat{u}_n)(\CC([0, T]; \mathcal{H}_n)) = 1, ~~\text{ i.e. } ~~\abs{\widehat{u}_n(t)}_{\mathcal{H}_n} = \abs{{u}_n(t)}_{\mathcal{H}_n}, ~~\PP\text{-a.s.}
\]

Using \eqref{same distribution} together with \eqref{manifold}, we get

\begin{align*}
    &\widehat{u}_{n} \to \widehat{u} ~~\text{in} ~~~\CC([0, T];L^2),\quad \text{and}\\
    &u_n(t) \in \MM
\end{align*}

Also, from the theorem \eqref{thm:4.1}, we have
\begin{align}
    &\sup_{1\leq n} \hat{\EE}\left( \sup_{p\in[0,T]} \norm{\widehat{u}_{n}(p)}_{\mathcal{V}}^2 \right) \leq K_1,\label{est1}  \\
    &\sup_{1\leq n} \hat{\EE}\left( \sup_{p\in[0,T]} \norm{\widehat{u}_{n}(p)}_{L^{2n}}^{2n} \right) \leq K_2,  \\
    &\sup_{1\leq n} \hat{\EE}\left( \int_0^T \norm{\widehat{u}_{n}(p)}_\mathcal{E}^2 dp \right) \leq K_3. \label{est3}
\end{align}
From \eqref{est3}, we can infer that   there exists a weakly convergent subsequence with the same label $\widehat{u}_{n}$ of $\widehat{u}_{n}$ in $L^2([0, T] \times \hat{\Omega}; \DD(\mathcal{A}))$. Also, by \eqref{same distribution} $\widehat{u}_{n} \to \widehat{u}$ in $\mathcal{X}_T$, $\hat{\PP}$-a.s, That is,
\begin{equation}\label{eq:7.5}
    \hat{\EE}\left( \int_0^T \norm{\widehat{u}(p)}_\mathcal{E}^2 dp \right) < \infty.
\end{equation}

Similarly, using \eqref{est1}, we have weak star convergent subsequence $\widehat{u}_{n}$ such that
\begin{equation}\label{eq:7.6}
    \hat{\EE}\left( \sup_{0\leq p\leq T} \norm{\widehat{u}(p)}^2_{\mathcal{V}} \right) < \infty.
\end{equation}

For all $n \geq 1$, we define a process $\widehat{M}_n$ having paths in $\CC([0, T]; \mathcal{H}_n)$, and in $\CC([0, T];L^2)$ by
\begin{equation}\label{eqM_n}
    \widehat{M}_n = \widehat{u}_{n}(t) - \widehat{u}_{n}(0) - \int_0^t \left(-\Delta^2\widehat{u}_n(p) + 2\Delta \widehat{u}_{n}(p)\right)dp - \int_0^t F(\widehat{u}_{n}(p))dp - \frac{1}{2} \sum_{k=1}^N \int_0^t m_k(\widehat{u}_{n}(p))dp - \sum_{k=1}^N \int_0^t B_k(\widehat{u}_{n}(p))dW_k.
\end{equation}

\begin{lemma}\label{lem:5.1}

$\widehat{M}_n$ is an square integrable martingale function with the quadratic variation $\ang{\ang{\widehat{M}_n}}_t$ defined as
\begin{align*}
    \ang{\ang{\widehat{M}_n}}_t = \int_0^t (d \widehat {M})^2 = \sum_{k=1}^N \int_0^t \abs{B_k(\widehat{u}_{n}(p))}^2 dp.
\end{align*}
\end{lemma}

\begin{proof}

As we know that  $u_n$ and $\widehat{u}_{n}$ have same laws, so for each $t_1, t_2 \in [0, T]$, such that $t_1 \leq t_2$, for each $\mathrm{h} \in \CC([0, T], \mathcal{H})$, and for all $\eta_1, \eta_2 \in \mathcal{H}$, one obtains,
\begin{equation}\label{eq:7.8}
    \hat{\EE}\left[ \ang{\widehat{M}_n(t_1) - \widehat{M}_n(t_2), \eta_1}_\mathcal{H} \mathrm{h}(\widehat{u}_{n}|_{[0,t_2]}) \right] = 0
\end{equation}
and
\begin{multline}\label{eq:7.9}
    \hat{\EE}\left[ \left( \ang{\widehat{M}_n(t_1), \eta_1}_\mathcal{H} \ang{\widehat{M}_n(t_1), \eta_2}_\mathcal{H} - \ang{\widehat{M}_n(t_2), \eta_1}_\mathcal{H} \ang{\widehat{M}_n(t_2), \eta_2}_\mathcal{H} \right) \mathrm{h}(\widehat{u}_{n}|_{[0,t_2]}) \right] \\
    - \hat{\EE}\left[ \left( \sum_{k=1}^m \int_{t_1}^{t_2} \ang{B_k(\widehat{u}_{n}(\tau))\eta_1, B_k(\widehat{u}_{n}(\tau))\eta_2}_\mathcal{H} d\tau \right) \cdot \mathrm{h}(\widehat{u}_{n}|_{[0,t_2]}) \right] = 0.
\end{multline}
We are done with the proof
\end{proof}

\begin{lemma}\label{5.2}
The process defined as
\begin{align*}
  \widehat{M}_n &= \widehat{u}_{n}(t) - \widehat{u}_{n}(0) - \int_0^t \left(-\Delta^2\widehat{u}_n(p) + 2\Delta \widehat{u}_{n}(p)\right)dp - \int_0^t F(\widehat{u}_{n}(p))dp \\&- \frac{1}{2} \sum_{k=1}^N \int_0^t m_k(\widehat{u}_{n}(p))dp - \sum_{k=1}^N \int_0^t B_k(\widehat{u}_{n}(p))dW_k.
\end{align*}
is $\mathcal{H}-$ valued continuous process.
\end{lemma}

\begin{proof}
By considering  $\widehat{u} \in \CC([0, T], \VV)$, it is enough to construct the following inequality
\begin{align}
    \left\| - \int_0^t \left(-\Delta^2\widehat{u}_n(p) + 2\Delta \widehat{u}_{n}(p)\right)dp - \int_0^t F(\widehat{u}_{n}(p))dp - \frac{1}{2} \sum_{k=1}^N \int_0^t m_k(\widehat{u}_{n}(p))dp - \sum_{k=1}^N \int_0^t B_k(\widehat{u}_{n}(p))dW_k \right\|_{L^2} < \infty.
\end{align}

To prove this lemma, we will take each term separately.

Consider  $\int_0^t \left(-\Delta^2\widehat{u}_n(p) + 2\Delta \widehat{u}_{n}(p)\right)dp= \int_0^t \mathcal{A}{u}_{n}(p)dp$.  Using Cauchy-Schwartz inequality together $\mathcal{V} \hookrightarrow \mathcal{H}$ and \eqref{Au} we have

\begin{align}
\hat{\EE}\left(\int_0^t \abs{\mathcal{A} \widehat{u}(p)}_{L^2} dp\right) &\leq \hat{\EE}\left(\int_0^T \abs{\mathcal{A} \widehat{u}(p)}_{L^2} dp\right)\notag \\
&\leq T^{\frac{1}{2}} \left(\hat{\EE}\left(\int_0^T \abs{\mathcal{A} \widehat{u}(p)}^2_{L^2} dp\right)\right)^{\frac{1}{2}} \notag \leq C^{\frac{1}{2}}T^{\frac{1}{2}}\left( \hat{\EE}\left(\int_0^T \norm{\widehat{u}(p)}^2_{\mathcal{E}} dp\right)\right)^{\frac{1}{2}}
\end{align}

Using \eqref{eq:5.3}, we have

\begin{equation}\label{eq:7.10}
\hat{\EE}\left(\int_0^t \abs{\mathcal{A} \widehat{u}(p)}_{L^2} dp\right) < \infty.
\end{equation}
Also,
\begin{align*}
    \hat{\EE}\left(\int_0^t \abs{F(\widehat{u}(p))}_{L^2} dp\right) &\leq \hat{\EE}\left(\int_0^T \mathcal{G}(\norm{\widehat{u}(p)}_\mathcal{V}, 0)\norm{\widehat{u}(p)}_\mathcal{V} dp\right) \notag \\
    &\leq \hat{\EE}\left(\int_0^T C_{7}\norm{\widehat{u}(p)}_\mathcal{V} dp\right) 
    \leq C_{7}T^{\frac{1}{2}} \hat{\EE}\left(\int_0^T \norm{\widehat{u}(p)}_\mathcal{V}^2 dp\right)^{\frac{1}{2}}.
\end{align*}
Using \eqref{eq:5.1}, it follows that
\begin{equation}\label{eq:7.11}
\hat{\EE}\left(\int_0^t \abs{F(\widehat{u}(p))}_{L^2} dp\right) < \infty.
\end{equation}

Next, using \eqref{m_K}, we can infer that

\begin{align*}
    \hat{\EE}\left(\int_0^t \abs{m_k(\widehat{u}(p)))} dp\right) \leq \hat{\EE}\left(\int_0^T C_{8} dp\right) = C_{8}T < \infty.
\end{align*}

\begin{equation}\label{eq:7.12}
\hat{\EE}\left(\int_0^t \abs{m_k(\widehat{u}(p)))} dp\right) < \infty.
\end{equation}

Finally, because $\int_0^t B_k(\widehat{u}(p))dW_k$ is stochastic integral so
\begin{equation}\label{eq:7.13}
\hat{\EE}\left(\int_0^t B_k(\widehat{u}(p))dW_k\right) = 0.
\end{equation}

Combining \eqref{eq:7.10}, \eqref{eq:7.11}, \eqref{eq:7.12} and \eqref{eq:7.13} we have
\begin{align*}
    \left\| - \int_0^t \left(-\Delta^2\widehat{u}_n(p) + 2\Delta \widehat{u}_{n}(p)\right)dp - \int_0^t F(\widehat{u}_{n}(p))dp - \frac{1}{2} \sum_{k=1}^N \int_0^t m_k(\widehat{u}_{n}(p))dp - \sum_{k=1}^N \int_0^t B_k(\widehat{u}_{n}(p))dW_k. \right\|_{L^2} < \infty.
\end{align*}
We are done with the proof.
\end{proof}

\begin{theorem}\label{5.3}
    For $t_1, t_2 \in [0, T]$ with $t_1 \leq t_2$ and $\eta \in H$, the following are true:
    \begin{align*}
       &1.~~ \lim\limits_{n\to\infty} \ang{\widehat{u}_n(t), \mathcal{Z}_n(\eta)}_\mathcal{H} = \ang{\widehat{u}(t), \eta}_\mathcal{H}, ~~\hat{\PP}-a.s.\\
    &2.~~\lim\limits_{n\to\infty} \int_{t_1}^{t_2} \ang{\mathcal{A}\widehat{u}_n(p), \mathcal{Z}_n(\eta)}_\mathcal{H} dp = \int_{t_1}^{t_2} \ang{\mathcal{A}\widehat{u}(p), \eta}_\mathcal{H} dp,~~ \hat{\PP}-a.s.\\
    &3.~~\lim\limits_{n\to\infty} \int_{t_1}^{t_2} \ang{F(\widehat{u}_n(p)), \mathcal{Z}_n(\eta)}_\mathcal{H} dp = \int_{t_1}^{t_2} \ang{F(\widehat{u}(p)), \mathcal{Z}(\eta)}_\mathcal{H} dp, ~~\hat{\PP}-a.s. ~~for~~ \eta \in H.\\
 &4.~~\lim\limits_{n\to\infty} \int_{t_1}^{t_2} \ang{m_k(\widehat{u}_n(p)), \mathcal{Z}_n(\eta)}_\mathcal{H} dp = \int_{t_1}^{t_2} \ang{m_k(\widehat{u}(p)), \mathcal{Z}(\eta)}_\mathcal{H} dp, ~~\hat{\PP}-a.s.\\
 &5.~~\lim\limits_{n\to\infty} \int_{t_1}^{t_2} \ang{B_k(\widehat{u}_n(p)), \mathcal{Z}_n(\eta)}_\mathcal{H} dW_k(p) = \int_{t_1}^{t_2} \ang{B_k(\widehat{u}(p)), \eta}_\mathcal{H} dW_k(p), ~~\hat{\PP}-a.s.
\end{align*}
\end{theorem}

\begin{proof}
    Using \eqref{same distribution}, we can say that $\widehat{u} _{n}$ converge to $\widehat{u}$ in $\mathcal{X}_T$ $\hat{\PP}$-a.s. \\
    But $\widehat{u} _{n}$ converge to $\widehat{u}$ in $C\left([0,T],\mathcal{H}\right)$ $\hat{\PP}$-a.s. and $\mathcal{Z}_{n}(\eta _{})$ converge to $\eta$ in $\mathcal{H}$ $\hat{\PP}$-a.s. Therefore, we have

\begin{align*}
\lim_{n\to\infty} \inner{\widehat{u}_n(t)}{\mathcal{Z}_n(\eta)} - \inner{\widehat{u}(t)}{\eta}_\mathcal{H} 
&= \lim_{n\to\infty} \inner{\widehat{u}_n(t) - \widehat{u}(t)}{\mathcal{Z}_n(\eta)}_\mathcal{H} + \lim_{n\to\infty} \inner{\widehat{u}(t)}{\mathcal{Z}_n(\eta) - \eta}_\mathcal{H} = 0, \quad \Phat\text{-a.s.}
\end{align*}
This proves (1).\\
Now,
\begin{align*}
\int_{t_1}^{t_2} \inner{\A(\widehat{u}_n(p))}{\Z_n(\eta)}_\H \,dp &- \int_{t_1}^{t_2} \inner{\A(\widehat{u}(p))}{\eta}_\H \,dp \\
&= \int_{t_1}^{t_2} \inner{\A(\widehat{u}_n(p)) - \A(\widehat{u}(p)))}{\eta}_\H \,dp + \int_{t_1}^{t_2} \inner{\A(\widehat{u}_n(p))}{\Z_n(\eta) - \eta}_\H \,dp \\
&= \int_{t_1}^{t_2} \inner{\A(\widehat{u}_n(p) - \widehat{u}(p)))}{\eta}_\H \,dp + \int_{t_1}^{t_2} \inner{\A(\widehat{u}_n(p))}{\Z_n(\eta) - \eta}_\H \,dp \\
\end{align*}
For fixed $\eta \in \H= L^2(\mathcal{O})$, the Riesz representation theorem \cite{Rudin} guarantees the existence of a unique $\phi_\eta \in D(\A)$ such that:
\[
\langle Au, \eta \rangle_{\H} = \langle u, \phi_\eta \rangle_{D(A)} \quad \forall u \in D(\A),
\]
Therefore,
\begin{align*}
\int_{t_1}^{t_2} \inner{\A(\widehat{u}_n(p))}{\Z_n(\eta)}_\H \,dp &- \int_{t_1}^{t_2} \inner{\A(\widehat{u}(p))}{\eta}_\H \,dp \\
    &\leq \int_{t_1}^{t_2} \inner{\widehat{u}_n(p) - \widehat{u}(p)}{\phi_\eta}_{D(\A)} \,dp + \int_{t_1}^{t_2} |\A(\widehat{u}_n(p))|_\H |\Z_n(\eta) - \eta|_\H \,dp \\
&\leq \int_{t_1}^{t_2} \inner{\widehat{u}_n(p) - \widehat{u}(p)}{\phi_\eta}_{D(\A)} \,dp + \int_0^T |\widehat{u}_n(p)|_{D(\A)} |\Z_n(\eta) - \eta|_\H \,dp \\
&\leq \int_{t_1}^{t_2} \inner{\widehat{u}_n(p) - \widehat{u}(p)}{\phi_\eta}_{D(\A)} \,dp + \left(\int_0^T |\widehat{u}_n(p)|^2_{D(\A)} \,dp\right)^{\frac{1}{2}} \left(\int_0^T |\Z_n(\eta) - \eta|^2_\H \,dp\right)^{\frac{1}{2}} \\
&= \int_{t_1}^{t_2} \inner{\widehat{u}_n(p) - \widehat{u}(p)}{\phi_\eta}_{D(\A)} \,dp + \left(\int_0^T |\widehat{u}_n(p)|^2_{D(\A)} \,dp\right)^{\frac{1}{2}} \left(|\Z_n(\eta) - \eta|^2_\H \int_0^T dp\right)^{\frac{1}{2}} \\
&= \int_{t_1}^{t_2} \inner{\widehat{u}_n(p) - \widehat{u}(p)}{\phi_\eta}_{D(\A)} \,dp + \|\widehat{u}_n\|_{L^2([0,T],D(\A))} |\Z_n(\eta) - \eta|_\H T^{\frac{1}{2}} \\
&= \int_{t_1}^{t_2} \inner{\widehat{u}_n(p) - \widehat{u}(p)}{\phi_\eta}_{D(\A)} \,dp + \|\widehat{u}_n\|_{L^2([0,T],D(\A))} |\Z_n(\eta) - \eta|_\H T^{\frac{1}{2}}
\end{align*}

As $\widehat{u}_n$, $\Phat$-a.s. converges weakly to limit $\widehat{u}$ in $L^2([0,T], \D(\A))$, also $\widehat{u}_n$ is uniformly bounded in $L^2([0,T], D(\A))$ and $\Z_n(\eta)$ converges to $\eta$ in $\H$-norm. It follows that
\[
\lim_{n\to\infty} \int_{t_1}^{t_2} \inner{\widehat{u}_n(p) - \widehat{u}(p)}{\phi_\eta}_{D(\A)} \,dp = 0,
\]
and
\[
\lim_{n\to\infty} |\Z_n(\eta) - \eta|_\H = 0,
\]
Therefore,
\begin{align*}
    \lim_{n\to\infty} \int_{t_1}^{t_2} \inner{\A(\widehat{u}_n(p))}{\Z_n(\eta)}_\H \,dp = \int_{t_1}^{t_2} \inner{\A(\widehat{u}(p))}{\eta}_\H \,dp,
\end{align*}
So (2) is proved.

Proving (3), let $\eta \in \V$
\begin{align*}
\int_{t_1}^{t_2} \inner{F(\widehat{u}_n(p))}{\Z_n(\eta)}_\H \,dp &- \int_{t_1}^{t_2} \inner{F(\widehat{u}(p))}{\eta}_\H \,dp \\
&= \int_{t_1}^{t_2} \inner{F(\widehat{u}_n(p)) - F(\widehat{u}(p)))}{\eta}_\H \,dp + \int_{t_1}^{t_2} \inner{F(\widehat{u}_n(p))}{\Z_n(\eta) - \eta}_\H \,dp \\
&\leq \int_0^T |F(\widehat{u}_n(p)) - F(\widehat{u}(p)))|_\H |\eta|_\H \,dp + \int_0^T |F(\widehat{u}_n(p))|_\H |\Z_n(\eta) - \eta|_\H \,dp.
\end{align*}

Using Lipschitz property together with  $\V \hookrightarrow \H$ we get,
\begin{align*}
\int_{t_1}^{t_2} \inner{F(\widehat{u}_n(p))}{\Z_n(\eta)}_\H \,dp &- \int_{t_1}^{t_2} \inner{F(\widehat{u}(p))}{\eta}_\H \,dp \\
&\leq \int_0^T (\mathcal{G}(\|\widehat{u}_n(p)\|_\V, \|\widehat{u}(p)\|_\V)\|\widehat{u}_n(p) - \widehat{u}(p)\|_\V)|\eta|_\H \,dp + \int_0^T |F(\widehat{u}_n(p))|_\H \|\Z_n(\eta) - \eta\|_\V \,dp \\
&\leq \left(\int_0^T (\mathcal{G}(\|\widehat{u}_n(p)\|_\V, \|\widehat{u}(p)\|_\V))^2 \,dp\right)^{\frac{1}{2}} \cdot \left(\int_0^T \|\widehat{u}_n(p) - \widehat{u}(p)\|_\V^2 \,dp\right)^{\frac{1}{2}} \\
&\quad + \left(\int_0^T |F(\widehat{u}_n(p))|^2_\H \,dp\right)^{\frac{1}{2}} \cdot \left(\int_0^T \|\Z_n(\eta) - \eta\|_\V^2 \,dp\right)^{\frac{1}{2}}
\end{align*}

Now, $\widehat{u}_n(p) \to \widehat{u}(p)$ strongly in $L^2([0,T], \V)$, bounded uniformly  $L^2([0,T], \V)$ and $\Z_n(\eta) \to \eta$ in $\V$. It follows
\[
\lim_{n\to\infty} \int_{t_1}^{t_2} \inner{F(\widehat{u}_n(p))}{\Z_n(\eta)}_\H \,dp - \int_{t_1}^{t_2} \inner{F(\widehat{u}(p))}{\eta}_\H \,dp = 0, \quad \Phat\text{-a.s.}
\]
We are done with (3).

Next, we prove (4). Let $\eta \in \H$,
\begin{align*}
\int_{t_1}^{t_2} \inner{m_k(\widehat{u}_n(p))}{\Z_n(\eta)}_\H \,dp &- \int_{t_1}^{t_2} \inner{m_k(\widehat{u}(p))}{\Z(\eta)}_\H \,dp \\
&= \int_{t_1}^{t_2} \inner{m_k(\widehat{u}_n(p)) - m_k(\widehat{u}(p)))}{\eta}_\H \,dp - \int_{t_1}^{t_2} \inner{m_k(\widehat{u}_n(p))}{\Z_n(\eta) + \eta}_\H \,dp, \\
&\leq \int_0^T |m_k(\widehat{u}_n(p)) - m_k(\widehat{u}(p)))|_\H |\eta|_\H \,dp + \int_0^T |m_k(\widehat{u}_n(p)|_\H |\Z_n(\eta) - \eta|_\H \,dp.
\end{align*}

Using Lipschitz property of $m_k(\widehat{u}_n(p))$ along with $\V \hookrightarrow \H$ and \eqref{m_K} we obtain
\begin{align*}
\int_{t_1}^{t_2} \inner{m_k(\widehat{u}_n(p))}{\Z_n(\eta)}_\H \,dp &- \int_{t_1}^{t_2} \inner{m_k(\widehat{u}(p))}{\Z(\eta)}_\H \,dp \\
&\leq \int_0^T C\|f_k\|_{\V}^2 \left[2 + \|\widehat{u}_n(p)\|_\V^2 + \|\widehat{u}(p)\|_\V^2 + (\|\widehat{u}_n(p)\|_\V + \|\widehat{u}(p)\|_\V)^2\right] \|\widehat{u}_n(p) - \widehat{u}(p)\|_\V |\eta|_\H \,dp \\
&\quad + \int_0^T |m_k(\widehat{u}_n(p)|_\H C \| \Z_n(\eta) - \eta \| \,dp, \\
&\leq \left[\int_0^T \left(C\|f_k\|_\V^2 \left[2 + \|\widehat{u}_n(p)\|_\V^2 + \|\widehat{u}(p)\|_\V^2 + (\|\widehat{u}_n(p)\|_\V + \|\widehat{u}(p)\|_\V)^2\right]\right)^2 |\eta|^2_\H \,dp\right]^{\frac{1}{2}} \\
&\quad \times \left[\int_0^T \|\widehat{u}_n(p) - \widehat{u}(p)\|_\V^2 \,dp\right]^{\frac{1}{2}}  + \left[\int_0^T |m_k(\widehat{u}_n(p))|^2_\H \right]^{\frac{1}{2}} \cdot \left[\int_0^T \| \Z_n(\eta) - \eta \|_\V^2 \,dp\right]^{\frac{1}{2}}
\end{align*}

By similar argumentation as used in (3),
\begin{align*}
&\left[\int_0^T \left(C\|f_k\|_\V^2 \left[2 + \|\widehat{u}_n(p)\|_\V^2 + \|\widehat{u}(p)\|_\V^2 + (\|\widehat{u}_n(p)\|_\V + \|\widehat{u}(p)\|_\V)^2\right]\right)^2 |\eta|^2_\H \,dp\right]^{\frac{1}{2}} \\
&\quad \times \left[\int_0^T \|\widehat{u}_n(p) - \widehat{u}(p)\|_\V^2 \,dp\right]^{\frac{1}{2}} + \left[\int_0^T |m_k(\widehat{u}_n(p))|^2_\H \right]^{\frac{1}{2}} \cdot \left[\int_0^T \| \Z_n(\eta) - \eta \|_\V^2 \,dp\right]^{\frac{1}{2}} \to 0
\end{align*}
We are done with (4).

Finally, we will prove the result (5). Suppose $\eta \in \H$ and  with following Lipschitz condition \cite{Hussain_2005},
\[
\|B_k(\widehat{u}_n(p) - B_k(\widehat{u}(p)))\|_\H \leq |f_k|_\H (|\widehat{u}_n(p)|_\H + |\widehat{u}(p)|_\H)|\widehat{u}_n(p) - \widehat{u}(p)|_\H.
\]
We have 
\begin{align*}
\int_{t_1}^{t_2} \inner{B_k(\widehat{u}_n(p))}{\Z_n(\eta)}_\H \,dW_j(p) &- \int_{t_1}^{t_2} \inner{B_k(\widehat{u}(p))}{\Z(\eta)}_\H \,dW_k(p) \\
&= \int_{t_1}^{t_2} \inner{B_k(\widehat{u}_n(p)) - B_k(\widehat{u}(p)))}{\eta}_\H \,dW_k(p) + \int_{t_1}^{t_2} \inner{B_k(\widehat{u}_n(p))}{\Z_n(\eta) - \eta}_\H \,dW_k(p) \\
&\leq \int_0^T |B_k(\widehat{u}_n(p)) - B_k(\widehat{u}(p))|_\H|\eta|_\H \,dW_k(p) + \int_0^T |B_k(\widehat{u}_n(p))|_\H|\Z_n(\eta) - \eta|_\H \,dW_k(p) \\
&\leq \int_0^T |f_k|_\H (|\widehat{u}_n(p)|_\H + |\widehat{u}(p)|_\H)|\widehat{u}_n(p) - \widehat{u}(p)|_\H |\eta|_\H \,dW_k(p) \\
&\quad + \int_0^T |B_k(\widehat{u}_n(p))|_\H|\Z_n(\eta) - \eta|_\H \,dW_k(p) \\
&= \mathcal{I}_{1} + \mathcal{I}_{2} 
\end{align*}

Using Ito Isometry, we have
\begin{align*}
\hat{\E}(\mathcal{I}_{1}^{2} )&=\hat{\E} \left(\left(\int_0^T |B_k(\widehat{u}_n(p)) - B_k(\widehat{u}(p)))|_\H|\eta|_\H \,dW_k(p)\right)^2\right) \\
&= \hat{\E}\left(\int_0^T |f_k|^2_\H (|\widehat{u}_n(p)|_\H + |\widehat{u}_h|_\H)^2 |\widehat{u}_n(p) - \widehat{u}(p)|^2_\H |\eta|^2_\H \,d(p)\right) \\
&\leq \hat{\E}\left[\left(\int_0^T |f_k|^4_\H (|\widehat{u}_n(p)|_\H + |\widehat{u}_h|_\H)^4 |\eta|^4_\H \,dp\right)^{\frac{1}{2}} \cdot \left(\int_0^T |\widehat{u}_n(p) - \widehat{u}(p)|^4_\H \,dp\right)^{\frac{1}{2}}\right] \\
&\leq \hat{\E}\left[\left(\int_0^T |f_k|^4_\H (|\widehat{u}_n(p)|_\H + |\widehat{u}_h|_\H)^4 |\eta|^4_\H \,dp\right)^{\frac{1}{2}} \cdot \|\widehat{u}_n(p) - \widehat{u}(p)\|^2_{L^4([0,T],\H)}\right].
\end{align*}

Using  $C([0,T], \H) \hookrightarrow L^4([0,T], \H)$, it follows that,
\begin{align*}
\hat{\E}(\mathcal{I}_1^2) &= \hat{\E}\left(\left(\int_0^T |B_k(\widehat{u}_n(p)) - B_k(\widehat{u}(p)))|_\H|\eta|_\H \,dW_k(p)\right)^2\right) \\
&\leq \hat{\E}\left[\left(\int_0^T |f_k|^4_\H (|\widehat{u}_n(p)|_\H + |\widehat{u}_h|_\H)^4 |\eta|^4_\H \,dp\right)^{\frac{1}{2}} \cdot C^2 \|\widehat{u}_n(p) - \widehat{u}(p)\|^2_{C([0,T],\H)}\right]
\end{align*}
since $\widehat{u}_n(p) \to \widehat{u}(p)$ in $C([0,T], \H)$ so
\[
\lim_{n\to\infty} \hat{\E}(\mathcal{I}_1^2) = \hat{\E}\left(\left(\int_0^T |B_k(\widehat{u}_n(p)) - B_k(\widehat{u}(p)))|_\H|\eta|_\H \,dW_k(p)\right)^2\right) = 0.
\]

Consider $\hat{\E}(\mathcal{I}_2^2)$:
\begin{align*}
\hat{\E}(\mathcal{I}_2^2) &= \hat{\E}\left(\left(\int_0^T c^2|B_k(\widehat{u}_n(p))|_\H|\Z_n(\eta) - \eta|_\H \,dW_k(p)\right)^2\right) \\
&= \hat{\E}\left(\int_0^T c^2|B_k(\widehat{u}_n(p))|^2_\H|\Z_n(\eta) - \eta|^2_\H \,dW_k(p)\right) \\
&= \hat{\E}\left(|\Z_n(\eta) - \eta|^2_\H \int_0^T c^2|B_k(\widehat{u}_n(p))|^2_\H \,dW_k(p)\right)
\end{align*}

As $\Z_n\eta \to \eta$ in $\H$ so $|\Z_n\eta - \eta|_\H \to 0$ giving
\[
\lim_{n\to\infty} \hat{\E}(\mathcal{I}_2^2) = \hat{\E}\left(|\Z_n(\eta) - \eta|^2_\H \int_0^T c^2|B_k(\widehat{u}_n(p))|^2_\H \,dW_k(p)\right) = 0
\]

Thus,
\[
\hat{\E}\left[\int_{t_1}^{t_2} \inner{B_k(\widehat{u}_n(p))}{\Z_n(\eta)}_\H \,dW_k(p) - \int_{t_1}^{t_2} \inner{B_k(\widehat{u}(p)))}{\Z(\eta)}_\H \,dW_k(p)\right] \leq \hat{\E}(\mathcal{I}_1^2) + \hat{\E}(\mathcal{I}_2^2) \to 0,~~~~~as ~~~n \to \infty
\]

Hence, (5) is true and we are done with the proof.

\end{proof}

\begin{lemma}

For each $t_1, t_2 \in [0, T]$ such that $t_2 \leq t_1$, for every  $\eta \in \V$ and for all $h \in C([0,T], \H)$, the following equation is true.
\begin{align}\label{main}
    \lim_{n\to\infty} \mathbb{E}\left[
\left\langle \widehat{M}_n(t_1) - \widehat{M}_n(t_2), \eta \right\rangle_H h(\widehat{u}_n|_{[0,t_2]})
\right]
= \mathbb{E}\left[
\left\langle \widehat{M}(t_1) - \widehat{M}(t_2), \eta \right\rangle_\H h(\widehat{u}|_{[0,t_2]})
\right]
\end{align}

\end{lemma}
\begin{proof}

Let $t_1, t_2 \in [0, T]$ with $t_2 \leq t_1$ and $\eta \in V$. Then using \eqref{eqM_n} we have
\begin{align}\label{MM}
\langle \widehat{M}_n(t_1) - \widehat{M}_n(t_2), \eta \rangle &= \langle \widehat{u}_n(t_1) - \widehat{u}_n(t_2), \Z_n\eta \rangle - \int_{t_2}^{t_1} \langle (-\Delta^{2}(\widehat{u}_n(\tau))+2\Delta(\widehat{u}_n(\tau))), \Z_n\eta \rangle_H d\tau \notag \\
&\quad - \int_{t_2}^{t_1} \langle F(\widehat{u}_n(\tau)), \Z_n\eta \rangle_H d\tau  - \frac{1}{2} \sum_{k=1}^N \int_{t_2}^{t_1} \langle m_k(\widehat{u}_n(\tau)), \Z_n\eta \rangle_H d\tau  - \sum_{k=1}^N \int_{t_2}^{t_1} \langle B_k(\widehat{u}_n(\tau)), \Z_n\eta \rangle d\tau 
\end{align}

Using Theorem \eqref{5.3} to  \eqref{MM} we have
\begin{equation}\label{5.17}
\lim_{n\to\infty} \langle \widehat{M}_n(t_1) - \widehat{M}_n(t_2), \eta \rangle = \langle \widehat{M}(t_1) - \widehat{M}(t_2), \eta \rangle \quad \widehat{\mathbb{P}}-a.s. 
\end{equation}

Let's prove \eqref{main}. As we know that $\widehat{u}_n \to \widehat{u}$ in $\mathcal{X}_T$ and thus $\widehat{u}_n \to \widehat{u}$ in $C([0, T], \H)$, which gives
\begin{equation}\label{5.18}
\lim_{n\to\infty} h(\widehat{u}_n|_{[0,t_2]}) = h(\widehat{u}|_{[0,t_2]}) \quad \widehat{\mathbb{P}}-a.s. 
\end{equation}
and
\[
\sup_{n\in\mathbb{N}} |h(\widehat{u}_n|_{[0,t_2]})|_{L^\infty} < \infty.\] 

We set
\begin{align}
    \eta_n(\omega) := \langle \widehat{M}_n(t_1, \omega) - \widehat{M}_n(t_2, \omega), \eta_1 \rangle_\H h(\widehat{u}_n|_{[0,t_2]}), \quad \omega \in \widehat{\Omega}.
\end{align}
We claim that
\begin{equation}\label{claim}
\sup_{1\leq n} \widehat{\mathbb{E}}(|\eta_n|^2) < \infty 
\end{equation}

Using  Cauchy-Schwarz inequality together with $\V \hookrightarrow \V'$ $\forall n \in \mathbb{N}$, we can find a positive constant $c$ such that
\begin{align}\label{same}
    \widehat{\mathbb{E}}\left[|\eta_n|^2\right] \leq c|h|^2_{L^\infty}|\eta_1|^2_\V \widehat{\mathbb{E}}[|\widehat{M}_n(t_1)|^2_\H + |\widehat{M}_n(t_2)|^2_\H].
\end{align}

But $\widehat{M}_n$ is a continuous process and martingale having quadratic variation given in \eqref{lem:5.1}, using Burkholder's inequality we get
\begin{equation}\label{bb}
\widehat{\mathbb{E}} \left[ \sup_{p\in[0,T]} |\widehat{M}_n(p)|^2_\H]\right] \leq c\widehat{\mathbb{E}}\left[ \sum_{k=1}^N \int_0^T |B_k(\widehat{u}_n(\tau))|^2_\H d\tau \right] 
\end{equation}

Using Equation \eqref{bb} together with \eqref{5.15} we get
\begin{align}\label{ff}
    \widehat{\mathbb{E}}\left[ \sup_{p\in[0,T]} |\widehat{M}_n(p)|^2_\H\right] \leq \widehat{\mathbb{E}}\left[C_{9}^2 T N\right] = C_{9}^2 T N < \infty
\end{align}

Thus, using \eqref{ff} together with \eqref{bb}, we can conclude that \eqref{claim} is true. So, $\{\eta_n\}$ is uniform integrable sequence and from \eqref{5.17}  we can deduce that $\eta_n$ converges $\mathbb{P}-a.s.$. Hence, the application of the Vitali Theorem gives the required proof.

\end{proof}

\begin{corollary}\label{cor5.5}
For all $t_1, t_2 \in [0, T]$ such that $t_2 \leq t_1$ one has
\[
\widehat{\mathbb{E}}\left[ \widehat{M}(t_1) - \widehat{M}(t_2) \big| \mathscr{F}_{t_1} \right] = 0.
\]
\end{corollary}

\begin{lemma}\label{5.6}
For every $t_1, t_2 \in [0, T]$ with $t_2 \leq t_1$ and $\forall \eta_1, \eta_2 \in V$ we have
\begin{align*}
\lim_{n\to\infty} \widehat{\mathbb{E}}&\left[
\left(\langle \widehat{M}_n(t_1), \eta_1 \rangle \langle \widehat{M}_n(t_1), \eta_2 \rangle - \langle \widehat{M}_n(t_2), \eta_1 \rangle \langle \widehat{M}_n(t_2), \eta_2 \rangle\right) h(\widehat{u}_n|_{[0,t_2]})
\right]\\
&= \widehat{\mathbb{E}}\left[
\left(\langle \widehat{M}(t_1), \eta_1 \rangle \langle \widehat{M}(t_1), \eta_2 \rangle - \langle \widehat{M}(t_2), \eta_1 \rangle \langle \widehat{M}(t_2), \eta_2 \rangle\right) h(\widehat{u}_n|_{[0,t_2]})
\right],
\end{align*}
where $h \in C([0, T], H)$ and $\langle \cdot, \cdot \rangle$ is the duality product between $\V$ and $\V'$.
\end{lemma}

\begin{proof}

Assuem that $t_1, t_2 \in [0, T]$ with $t_2 \leq t_1$ and $\forall \eta_1, \eta_2 \in \V$. 
We set
\begin{align*}
\eta_n(\omega) &:= \left(\langle \widehat{M}_n(t_1, \omega), \eta_1 \rangle \langle \widehat{M}_n(t_1, \omega), \eta_2 \rangle - \langle \widehat{M}_n(t_2, \omega), \eta_1 \rangle \langle \widehat{M}_n(t_2, \omega), \eta_2 \rangle\right) h(\widehat{u}_n|_{[0,t_2]}(\omega)), \\
\eta(\omega) &:= \left(\langle \widehat{M}(t_1, \omega), \eta_1 \rangle \langle \widehat{M}(t_1, \omega), \eta_2 \rangle - \langle \widehat{M}(t_2, \omega), \eta_1 \rangle \langle \widehat{M}(t_2, \omega), \eta_2 \rangle\right) h(\widehat{u}|_{[0,t_2]}(\omega)), \quad \omega \in \widehat{\Omega}.
\end{align*}

Using \eqref{5.17} and \eqref{5.18} we have
\[
\lim_{n\to\infty} \eta_n(\omega) = \eta(\omega), \quad \widehat{\mathbb{P}}-a.s.
\]

We claim that for $a > 1$, the following result holds
\begin{equation}\label{4}
\sup_{1\leq n} \widehat{\mathbb{E}} [|\eta_n|^a] < \infty 
\end{equation}

We have same inequality as in previous lemma \eqref{same}
\begin{equation}\label{1}
\widehat{\mathbb{E}} [|\eta_n|^a] \leq c|h|^a_{L^\infty}|\eta_1|^a_V |\eta_2|^a_V \widehat{\mathbb{E}}\left[ |\widehat{M}_n(t_1)|_\H^{2a} + |\widehat{M}_n(t_2)|_\H^{2a} \right].
\end{equation}

Again, using Burkholder's inequality Theorem, we get
\begin{equation}\label{2}
\widehat{\mathbb{E}}\left[ \sup_{p\in[0,T]} |\widehat{M}_n(p)|_\H^{2a} \right] \leq c\widehat{\mathbb{E}}\left[ \left( \sum_{k=1}^N \int_0^T |B_k(\widehat{u}_n(\tau))|_\H^2 d\tau \right)^a \right]. 
\end{equation}

Also,
\begin{equation}\label{3}
\widehat{\mathbb{E}}\left[ \left( \sum_{k=1}^N \int_0^T |B_k(\widehat{u}_n(\tau))|_\H^2 d\tau \right)^a \right] \leq \widehat{\mathbb{E}}\left[ \left( N C_{9}^2 T \right)^a \right] = N^a C_{9}^{2a} T^a < \infty.
\end{equation}

From \eqref{1}, \eqref{2} and \eqref{3} we say  that \eqref{4} is true . And by using  Vitali Theorem , we have
\begin{align*}
    \lim_{n\to\infty} \widehat{\mathbb{E}}(\eta_n) = \widehat{\mathbb{E}}(\eta).
\end{align*}
\end{proof}

\begin{lemma}\label{5.7}
For each $t_1, t_2 \in [0, T]$, $\forall \eta_1, \eta_2 \in \V$ and $\forall h \in C([0, T], \H)$ we have
\begin{align*}
\lim_{n\to\infty} \widehat{\mathbb{E}}&\left[
\left(
\sum_{k=1}^N \int_{t_2}^{t_1} \langle (B_k(\widehat{u}_n(\tau)))^*\Z_n\eta_1, (B_k(\widehat{u}_n(\tau)))^*\Z_n\eta_2 \rangle_\R d\tau
\right) \cdot h(\widehat{u}_n|_{[0,t_2]})
\right]\\
&= \widehat{\mathbb{E}}\left[
\left(
\sum_{k=1}^N \int_{t_2}^{t_1} \langle (B_k(\widehat{u}(\tau)))^*\eta_1, (B_k(\widehat{u}(\tau)))^*\eta_2 \rangle_\R d\tau
\right) \cdot h(\widehat{u}|_{[0,t_2]})
\right].
\end{align*}
\end{lemma}

\begin{proof}
Assume that  $\eta_1, \eta_2 \in V$, we define
\begin{align*}
\eta_n = \left(
\sum_{k=1}^N \int_{t_2}^{t_1} \langle (B_k(\widehat{u}_n(\tau,\omega)))^*\Z_n\eta_1, (B_k(\widehat{u}_n(\tau,\omega)))^*\Z_n\eta_2 \rangle_\R d\tau
\right) \cdot h(\widehat{u}_n|_{[0,t_2]}).
\end{align*}

We claim that $\eta_n$ is uniformly integrable which converges in $\widehat{\mathbb{P}}$-a.s. to some $\eta$. 
To prove our claim, it is sufficient to prove that  $\exists a > 1$ such that,
\begin{equation}\label{claimmm}
\sup_{1\leq n} \widehat{\mathbb{E}}[|\eta_n|^a] < \infty. 
\end{equation}

Using the Cauchy-Schwarz inequality, we have 
\[
|(B_k(\widehat{u}_n(\tau,\omega)))^*\Z_n\eta_1|_\R \leq |B_k(\widehat{u}_n(\tau,\omega))| \cdot |\Z_n\eta_1|_\R \leq C_{9}|\eta_1|_\H.
\]

By using  Hölder's inequality, we get
\begin{align*}
\widehat{\mathbb{E}}[|\eta_n|^a] &= \widehat{\mathbb{E}}\left[
\left|
\left(
\sum_{k=1}^N \int_{t_2}^{t_1} \langle (B_k(\widehat{u}_n(\tau)))^*\Z_n\eta_1, (B_k(\widehat{u}_n(\tau)))^*\Z_n\eta_2 \rangle_\R d\tau
\right) \cdot h(\widehat{u}_n|_{[0,t_2]})
\right|^a
\right] \\
&\leq |h|^a_{L^\infty} \widehat{\mathbb{E}}\left[
\left(
\sum_{k=1}^N \int_0^T |(B_k(\widehat{u}_n(\tau)))^*\Z_n\eta_1| \cdot |(B_k(\widehat{u}_n(\tau)))^*\Z_n\eta_2| d\tau
\right)^a
\right] \\
&\leq (N C_{9}^2)^a |h|^a_{L^\infty} |\eta_1|^a_\H |\eta_2|^a_\H T^a.
\end{align*}

Therefore, our claim \eqref{claimmm} is true. And we can say that $\eta_n$ is uniformly integrable and converges $\widehat{\mathbb{P}}$-a.s. to some $\eta$.

Next, we will show the point-wise convergence. Assume that  $\omega \in \widehat{\Omega}$:
\begin{align}\label{prv}
\lim_{n\to\infty} \int_{t_2}^{t_1} \sum_{k=1}^N \langle (B_k(\widehat{u}_n(\tau)))^*\Z_n\eta_1, (B_k(\widehat{u}_n(\tau)))^*\Z_n\eta_2 \rangle_R d\tau
= \int_{t_2}^{t_1} \sum_{k=1}^N \langle (B_k(\widehat{u}(\tau)))^*\eta_1, (B_k(\widehat{u}(\tau)))^*\eta_2 \rangle_R d\tau. 
\end{align}

Fix $\omega \in \widehat{\Omega}$:
\begin{enumerate}
\item $\widehat{u}_n(\cdot, \omega) \to \widehat{u}(\cdot, \omega)$
\item The sequence $\widehat{u}_n(\cdot, \omega)_{n\geq 1}$ is uniformly bounded in $L^2([0, T], \V)$.
\end{enumerate}

In order to prove \eqref{prv}, it is sufficient to prove the following 

\begin{align}
    (B_k(\widehat{u}_n(\tau, \omega)))^*\Z_n\eta_1 \to (B_k(\widehat{u}(\tau, \omega)))^*\eta_1 \quad \text{in } ~~~L^2([t_2, t_1], \mathbb{\R}).
\end{align}

Using the Cauchy-Schwarz inequality:
\begin{align}
&\int_{t_2}^{t_1} |(B_k(\widehat{u}_n(\tau, \omega)))^*\Z_n\eta_1 - (B_k(\widehat{u}(\tau, \omega)))^*\eta_1|_\R^2 d\tau \notag\\
&= \int_{t_2}^{t_1} |(B_k(\widehat{u}_n(\tau, \omega)))^*(\Z_n\eta_1 - \eta_1) + (B_k(\widehat{u}_n(\tau, \omega)) - B_k(\widehat{u}(\tau, \omega)))^*\eta_1|_\R^2 d\tau \notag \\
&\leq 2\int_{t_2}^{t_1} |B_k(\widehat{u}_n(\tau, \omega))|^2 |\Z_n\eta_1 - \eta_1|_\H^2 d\tau + 2\int_{t_2}^{t_1} |B_k(\widehat{u}_n(\tau, \omega)) - B_k(\widehat{u}(\tau, \omega))|^2 |\eta_1|_\H^2 d\tau 
\end{align}

We set
\begin{align*}
A^1_n &= \int_{t_2}^{t_1} |B_k(\widehat{u}_n(\tau, \omega))|^2 |\Z_n\eta_1 - \eta_1|_\H^2 d\tau \\
A^2_n &= \int_{t_2}^{t_1} |B_k(\widehat{u}_n(\tau, \omega)) - B_k(\widehat{u}(\tau, \omega))|^2 |\eta_1|_\H^2 d\tau.
\end{align*}

But
\begin{align}\label{Z}
\lim_{n\to\infty} |\Z_n\eta_1 - \eta_1|_\H = 0 
\end{align}
and
\begin{align}\label{Bk}
|B_k(\widehat{u}_n)| \leq C_{9}, 
\end{align}
Using \eqref{Z} and \eqref{Bk} it follows that 

\begin{align}
    \lim_{n\to\infty} A^1_n = 0
\end{align}

Also, from \cite{Hussain_2005},
\begin{align*}
    |B_k(\widehat{u}_n) - B_k(\widehat{u})|_\H \leq \|f_k\|_\V(\|\widehat{u}_n\|_\V + \|\widehat{u}\|_\V)\|\widehat{u}_n - \widehat{u}\|_\V
\end{align*}
Thus,
\begin{align*}
    \int_{t_2}^{t_1} |B_k(\widehat{u}_n(\tau, \omega)) - B_k(\widehat{u}(\tau, \omega))|^2 |\eta_1|_\H^2 d\tau
\leq |\eta_1|_\H \int_{t_2}^{t_1} \|f_k\|_\V(\|\widehat{u}_n(\tau, \omega)\|_\V + \|\widehat{u}(\tau, \omega)\|_\V)\|\widehat{u}_n(\tau, \omega) - \widehat{u}(\tau, \omega)\|_\V d\tau
\end{align*}

Using \eqref{same distribution} and $\widehat{u}_n(\cdot, \omega) \to \widehat{u}(\cdot, \omega)$, we get
\[
\lim_{n\to\infty} A^2_n = 0.
\]
We are done with the lemma.
\end{proof}

By applying Lemma \eqref{main}, we can pass the limit in \eqref{eq:7.8}. Using Lemmas \eqref{5.6} and \eqref{5.7}, we can pass the limit in \eqref{eq:7.9}. After passing limits, we conclude that
\[
\widehat{\mathbb{E}}\left[ \langle \widehat{M}(t_1) - \widehat{M}(t_2), \eta_1 \rangle_\H h(\widehat{u}|_{[0,t_2]}) \right] = 0.
\]
and
\begin{align*}
\widehat{\mathbb{E}} &\left[
\left(
\langle \widehat{M}(t_1), \eta_1 \rangle_H \langle \widehat{M}(t_1), \eta_2 \rangle_\H 
- \langle \widehat{M}(t_2), \eta_1 \rangle_\H \langle \widehat{M}(t_2), \eta_2 \rangle_\H
\right) h(\widehat{u}|_{[0,t_2]})
\right] \\
&- \widehat{\mathbb{E}}\left[
\left(
\sum_{k=1}^m \int_{t_2}^{t_1} \langle B_k(\widehat{u}(\tau))\eta_1, B_k(\widehat{u}(\tau))\eta_2 \rangle_\H d\tau
\right) h(\widehat{u}|_{[0,t_2]})
\right] = 0.
\end{align*}
Consequently, we have the following immediate corollary.

\begin{corollary}
For each $t\in [0,T]$ we have
\begin{align*}
\left \langle \left \langle\widehat{M} \right \rangle \right \rangle _t = \int_0^t \sum_{k=1}^N \|B_k(\widehat{u}(p))\|_\H^2 dp.
\end{align*}
\end{corollary}

\section{ \textbf{\Large Existence of Martingale Solution}}
This section aims to state our main result of the paper.

\begin{theorem}\label{6.1}
There is a martingale solution to problem \eqref{main_eq_st_Ito}.
\end{theorem}

\begin{proof}
We will use the same procedure as Da Prato and Zabczyk in \cite{Da Prato} and Gaurav Dhariwal in \cite{ref2}. 
By using the leamma \eqref{5.2} and corollary \eqref{cor5.5}, we can say that we have  $\H$-valued continuous square integrable processes $\widehat{M}(t)$, $t \in [0, T]$, which is $(\mathscr{F}_t)-$ martingale. In addition, its quadratic variation is given as

\begin{align*}
   \left \langle \left \langle\widehat{M} \right \rangle \right \rangle _t = \int_0^t \sum_{k=1}^N \|B_k(\widehat{u}(p))\|_\H^2 dp.
\end{align*}

Thus, referring to  the Martingale Representation Theorem \eqref{thm:martingale_representation}, we can infer that we have 
\begin{enumerate}
\item A stochastic basis $(\widehat{\widehat\Omega}, \widehat{\mathscr{\widehat F}}, (\widehat{\mathscr{ \widehat F}}_t), \widehat{\mathbb{ \widehat P}})$
\item An $\mathbb{R}^N$-valued $\widehat{\mathscr{ \widehat F}}$-Wiener process $\widehat{ \widehat W}(t)$
\item A progressively measurable process $\widehat{\widehat u}$ such that for all $t \in [0, T]$ and $\omega \in V$ satisfies:
\begin{align*}
\langle \widehat{ \widehat u}(t), \omega \rangle - \langle \widehat{ \widehat u}(0), \omega \rangle = \int_0^t \left\langle -\Delta^{2}\widehat{ \widehat u}(p)+2 \Delta \widehat{\widehat u}(p)+F(\widehat{\widehat u}(p)) + \frac{1}{2}\sum_{k=1}^{N} m_{k}(\widehat{\widehat u}(p)), \omega\right\rangle dp + \sum_{k=1}^N \int_0^t \langle B_k(\widehat{\widehat u}), \omega \rangle d\widehat {\widehat W}_k
\end{align*}
\end{enumerate}

Hence, we are done with our proof that there is a martingale solution to the problem \eqref{thm:martingale_representation}
\end{proof}

\end{document}